\documentclass[a4paper, 12pt]{amsart}
\usepackage{amsmath, mathrsfs, amssymb, array,fancybox, upgreek}
\usepackage[all, matrix, arrow, curve]{xy}

\usepackage{enumitem}
\usepackage[subnum]{cases}
\usepackage{amsthm}
\usepackage{hyperref}

\usepackage{blkarray}

\makeatletter
\@addtoreset{equation}{subsection}
\makeatother

\newcommand{\n}{\operatorname{n}}
\newcommand{\coker}{\operatorname{coker}}
\newcommand{\Cl}{\operatorname{Cl}}

\newcommand{\Diff}{\operatorname{Diff}}
\newcommand{\Supp}{\operatorname{Supp}}
\newcommand{\Sing}{\operatorname{Sing}}
\newcommand{\red}{\operatorname{red}}
\newcommand{\const}{\operatorname{const}}

\newcommand{\wt}{\operatorname{wt}}
\newcommand{\ord}{\operatorname{ord}}

\newcommand{\Ho}{\operatorname{H}}

\newcommand{\len}{\operatorname{len}}
\newcommand{\gr}{\operatorname{gr}}

\newcommand{\pr}{\operatorname{pr}}

\newcommand{\unit}{\operatorname{unit}}

\newcommand{\qq}{\mathbin{\sim_{\scriptscriptstyle{\QQ}}}}

\newcommand{\ovalh}[1]{{\raisebox{2pt}{\ovalbox{#1}}}}
\newcommand{\ovalv}[1]{{\raisebox{7pt}{\ovalbox{#1}}}}

\newcommand{\toplus}{\mathbin{\tilde\oplus}}
\newcommand{\totimes}{\mathbin{\tilde\otimes}}

\renewcommand{\emptyset}{\varnothing}

\newcommand{\sigmaord}{\sigma\mbox{-}\ord}

\newcommand{\CC}{\mathbb{C}}
\newcommand{\ZZ}{\mathbb{Z}}
\newcommand{\PP}{\mathbb{P}}
\newcommand{\QQ}{\mathbb{Q}}

\newcommand{\mm}{{\mathfrak{m}}}

\newcommand{\KKK}{{\mathscr{K}}}

\newcommand{\EEE}{{\mathscr{E}}}
\newcommand{\OOO}{\mathscr{O}}
\newcommand{\DDD}{\mathscr{D}}
\newcommand{\CCC}{\mathscr{C}}
\newcommand{\AAA}{\mathscr{A}}
\newcommand{\BBB}{\mathscr{B}}
\newcommand{\FFF}{\mathscr{F}}
\newcommand{\GGG}{\mathscr{G}}
\newcommand{\VVV}{\mathscr{V}}

\newcommand{\muu}{{\boldsymbol{\mu}}}

\newcommand{\type}[1]{$\mathrm{#1}$}

\renewcommand\labelenumi{{\rm (\roman{enumi})}}
\renewcommand\theenumi{(\roman{enumi})}

\swapnumbers
\theoremstyle{plain}
\newtheorem{theorem}[subsection]{Theorem}
\newtheorem{lemma}[subsection]{Lemma}
\newtheorem{proposition}[subsection]{Proposition}

\newtheorem{scorollary}[equation]{Corollary}
\newtheorem*{claim*}{Claim}
\newtheorem{sclaim}[equation]{Claim}
\newtheorem{slemma}[equation]{Lemma}

\theoremstyle{definition}
\newtheorem{setup}[subsection]{Set-up}

\newtheorem{sdefinition}[equation]{Definition}
\newtheorem{example-remark}[subsection]{Remark-Example}
\newtheorem{subexample-remark}[equation]{Remark-Example}
\newtheorem{case}[subsection]{}
\newtheorem{scase}[equation]{}

\newtheorem{example}[subsection]{Example}

\newtheorem{remark}[subsection]{Remark}

\newtheorem{sremark}[equation]{Remark}

\newtheorem{assumption}[subsection]{Assumption}

\newtheorem{sconstruction}[equation]{Construction}

\newtheorem{computation}[subsection]{Computation}

\newcommand{\xref}[1]{{\rm\ref{#1}}}

\title{Threefold extremal contractions \\ of type \type{(IIA)}, II}
\author{Shigefumi Mori}

\address{
Shigefumi~Mori:
Research Institute for Mathematical Sciences,
Kyoto University, Kyoto, Japan
\newline\indent
Kyoto University Institute for Advanced Study,
Kyoto University, Kyoto, Japan
}
\email{mori@kurims.kyoto-u.ac.jp}

\author{Yuri Prokhorov}

 \thanks{The first author's work partially supported by JSPS KAKENHI Grant Numbers (B) 25287005 and (S) 24224001.
 The second author's work partially supported by the RFFI grants
 15-01-02164a,  15-01-02158a,  and by the Russian Academic Excellence Project '5-100'.}

\address{Yuri~Prokhorov:
Steklov Mathematical Institute of Russian Academy of Sciences, Moscow, Russia
\newline\indent
Department of Algebra, 
Moscow State Lomonosov University
\newline\indent
National Research University Higher School of Economics
}
\email{prokhoro@mi.ras.ru}

\begin{document}

% \begin{flushright}
% \texttt{\todayy}
% \end{flushright}
\begin{abstract}
Let $(X, C)$ be a germ of a threefold $X$ with terminal singularities
along an irreducible reduced complete curve $C$
with a contraction $f: (X, C)\to (Z, o)$
such that $C=f^{-1}(o)_{\red}$ and $-K_X$
is ample. 
% Assume that $(X, C)$ contains a point
% of type \type{(IIA)}.
% We complete the classification of such germs
% in terms of a general member
% $H\in |\OOO_X|$ containing $C$.
This paper continues our study of such germs containing a point
of type \type{(IIA)} started in \cite{Mori-Prokhorov-IIA-1}.
\end{abstract}
\maketitle

% \tableofcontents

\section{Introduction}
Let $(X,C)$ be a germ of a threefold with terminal singularities
along a reduced complete curve. We say that $(X,C)$
is an \textit{extremal curve germ} if
there is a contraction $f: (X,C)\to (Z,o)$ such that
$C=f^{-1}(o)_{\red}$ and $-K_X$
is $f$-ample.
Furthermore, $f$ is called \textit{flipping} if
its exceptional locus coincides with $C$ and
\textit{divisorial}
if its exceptional locus is two-dimensional.
If $f$ is not birational, then $\dim Z=2$ and
$(X,C)$ is said to be
a \textit{$\QQ$-conic bundle germ} \cite{Mori-Prokhorov-2008}.

In this paper we consider only extremal curve germs with \textit{irreducible}
central fiber $C$.
All the possibilities for the local behavior of $(X,C)$ are classified into types
\type{(IA)}, \type{(IC)}, \type{(IIA)}, \type{(IIB)},
\type{(IA^\vee)}, \type{(II^\vee)}, \type{(ID^\vee)}, \type{(IE^\vee)},
and \type{(III)},
whose definitions we
refer the reader to \cite{Mori-1988} and \cite{Mori-Prokhorov-2008}.

In this paper we complete the classification of extremal curve germs
containing points of type \type{(IIA)} started in \cite{Mori-Prokhorov-IIA-1}.
As in \cite{Kollar-Mori-1992}, \cite{Mori-Prokhorov-IA}, and
\cite{Mori-Prokhorov-IC-IIB}
the classification is done in terms of
a general hyperplane section,
that is, a general divisor $H$ of $|\OOO_X|_C$,
the linear subsystem of $|\OOO_X|$ consisting of sections containing $C$.
The case where $H$ is normal was treated in \cite{Mori-Prokhorov-IIA-1}.
In this paper we consider the case of non-normal $H$.
Our main result is the following.
\begin{theorem}\label{main}
Let $(X,C)$ be an extremal curve germ
and let $f: (X, C)\to (Z,o)$ be the corresponding
contraction.
Assume that $(X,C)$ it has a point $P$ of type \type{(IIA)}.
Furthermore, assume that the general member $H\in |\OOO_X|_C$ is not normal.
Then the following are the only possibilities
for the dual graph of $(H,C)$, and all the possibilities do occur.

\begin{enumerate}
\renewcommand\labelenumi{{\rm (\arabic{section}.\arabic{subsection}.\arabic{enumi})}\refstepcounter{equation}}
\renewcommand\theenumi{(\arabic{section}.\arabic{subsection}.\arabic{enumi})}

\item
\label{main-theorem-divisorial}
$f$ is divisorial\footnote{This case was erroneously
omitted in \cite[Th. 3.6 and Cor. 3.8]{Tziolas2005}.}, $f(H)\ni o$ is of type \type{D_{5}},
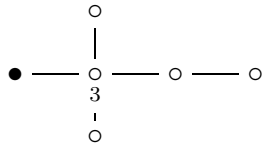
\begin{equation*}
\xy
\xymatrix@R=7pt@C=17pt{
&\circ\ar@{-}[d]
\\
\underset {} \bullet \ar@{-}[r]
&\underset 3\circ\ar@{-}[r]&\circ\ar@{-}[r]&\circ
\\
&\circ\ar@{-}[u]
}
\endxy
\end{equation*}

\item
\label{main-theorem-conic-bundle}
$f$ is a $\QQ$-conic bundle over a smooth surface,
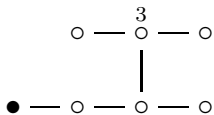
\begin{equation*}
\vcenter{
\xy
\xymatrix@R=7pt@C=11pt{
&\circ\ar@{-}[r]&\overset {3}\circ\ar@{-}[d]\ar@{-}[r]&\circ
\\
\bullet \ar@{-}[r] &\underset {}\circ\ar@{-}[r]&\circ\ar@{-}[r]&\circ
}
\endxy}
\end{equation*}
\end{enumerate}
In both cases
$X$ can have at most one extra  point of type \type{(III)}.
% All the possibilities occur.
\end{theorem}

\begin{remark}
If $(X,C)$ is an extremal curve germ of type \type{(IIA)}, then 
according to \cite[Corollary 2.6]{Mori-Prokhorov-IIA-1}
the general member $H\in |\OOO_X|_C$ is not normal if and only if $\Ho^0(\gr_C^1\OOO)=0$.
\end{remark}

Note that the description of a member $H\in |\OOO_X|_C$ is just
a part of our results. We also describe the infinitesimal structure
of the corresponding extremal curve germs.
Refer to \eqref{equation-possibilities-lP=3+III-b} and \ref{scorollary-9-6-8}
for the case \ref{main-theorem-divisorial} and 
to \eqref{equation-lP=4-gr-2-C-O} and \ref{case-conic-bundle}
for the case \ref{main-theorem-conic-bundle}.
We also provide many examples (see \ref{example-divisorial-lP=3}, \ref{example-divisorial-lP=7},
\ref{example-conic-bundle-lP=4+III}, \ref{example-conic-bundle-lP=8}).

The proof of the main theorem splits into cases
according to the invariant $\ell(P)$ which, in our case, can take values $\ell(P)\in\{ 3,\, 4,\, 7,\, 8\}$ 
(see \ref{equation-iP} and Proposition \ref{proposition-cases-lP}).
Cases of odd and even $\ell(P)$ 
will be considered in Sections \ref{section-lP=3-III} and
\ref{section-lP=4}, respectively.

\section{Preliminaries}\label{section-Preliminaries}
\begin{setup}
\label{Set-up}
Let $(X,C)$ be an extremal
curve germ and let $f: (X, C)\to (Z,o)$ be the corresponding
contraction.
The ideal sheaf of $C$ in $X$ we denote by $I_C$ or simply by $I$.
Assume that $(X,C)$ has a point $P$ of type \type{(IIA)}.
Then by \cite[6.7, 9.4]{Mori-1988} and
\cite[8.6, 9.1, 10.7]{Mori-Prokhorov-2008} $P$ is the only non-Gorenstein point of $X$
and $(X,C)$ has at most one Gorenstein singular point $R$ \cite[6.2]{Mori-1988},
\cite[9.3]{Mori-Prokhorov-2008}.
If $\Ho^0(\gr_C^1\OOO)=0$, then $(X,C)$ is not flipping 
\cite[ch. 7]{Kollar-Mori-1992}.
\end{setup}

\begin{scase}\label{sde}
Thus, in the case $\Ho^0(\gr_C^1\OOO)=0$, we have two possibilities:
\begin{itemize}
\item
$f$ is a $\QQ$-conic bundle and $(Z,o)$ is smooth \cite[Th. 1.2]{Mori-Prokhorov-2008};
\item
$f$ is a divisorial contraction and $(Z,o)$ is a cDV point
(or smooth) \cite[Th. 3.1]{Mori-Prokhorov-IA}. 
\end{itemize}
\end{scase}

\begin{case}\label{equation-iP}
Everywhere in this paper $(X,P)$
denotes a terminal singularity $(X,P)$ of type \type{cAx/4} and 
$(X^\sharp, P^\sharp)\to (X,P)$
denotes its index-one cover. 
% For any object $V$ on $X$
% we denote by $V^\sharp$ the pull-back of $V$ on $X^\sharp$.
Let
\begin{equation*}
\ell(P):=\len_P I^{\sharp (2)}/I^{\sharp 2},
\end{equation*}
where $I^\sharp$ is the ideal defining $C^\sharp$ in $X^\sharp$.
Recall (see \cite[(2.16)]{Mori-1988}) that in our case
\begin{equation*}
i_P(1)=\lfloor(\ell(P)+6)/4\rfloor.
\end{equation*}
% In our proof of Theorem \ref{main} we distinguish cases according
% to the value of $\ell(P)$ (see \ref{proposition-cases-lP}) and treat these cases separately
% in the next sections.
\end{case}

\begin{case}
\label{equation-IIA-point}
According to \cite[A.3]{Mori-1988} we can express the \type{(IIA)} point as
\begin{equation}
\label{equation-XC}
\begin{split}
(X, P)&=
\{\alpha=0\}/\muu_4(1, 1, 3, 2)\subset\CC^4_{y_1,\dots, y_4}/\muu_4(1, 1, 3, 2),
\\
C&=\{y_1\text{-axis}\}/\muu_4,
\end{split}
\end{equation}
where $\alpha=\alpha(y_1,\dots, y_4)$ is a semi-invariant such that
\begin{equation}\label{equation-alpha}
\wt\alpha\equiv 2\mod 4,\qquad \alpha\equiv y_1^{\ell(P)}y_j\mod (y_2, y_3, y_4)^2,
\end{equation}
where $j= 2$ (resp. $3$, $4$) if $\ell(P)\equiv 1$ (resp. $3$, $0$) $\mod 4$
\cite[(2.16)]{Mori-1988} and $(I^\sharp)^{(2)}=(y_j)+(I^\sharp)^{2}$.
Moreover, $y_2^2,\, y_3^2\in \alpha$ (because $(X,P)$ is of type
\type{cAx/4}).
\end{case}

\begin{case}\label{ge}
Recall that in our case
the general member $D\in |-K_X|$ does not contain $C$
\cite[Th. 7.3]{Mori-1988},
\cite[Prop. 1.3.7]{Mori-Prokhorov-2008}.
Hence $D\cap C=\{P\}$, $D\simeq f(D)$, and $D$ has at $P$ a singularity of type \type{D_{2n+1}}
\cite[6.4B]{Reid-YPG1987}.
In the coordinates $y_1,\dots,y_4$, the divisor $D$ is given by
\begin{equation*}
D=\{y_1= \xi\}/\muu_4,\qquad \xi\in (y_2,\, y_3,\, y_4).
\end{equation*}
\end{case}

\begin{scase}\label{sde1}
Let $H$ be a general member of $|\OOO_X|_C$ through $C$ and let $\beta\in \Ho^0(I_C)$ be a
non-zero section defining $H$.
Let $H_Z=f(H)$ and let $\psi: H^{\n}\to H$ be the normalization. The composition map
$H^{\n}\to H_Z$ has connected fibers. Moreover,
it is a rational curve fibration if $\dim Z=2$
and it is a birational contraction to a point $(H_Z, o)$
which is either smooth or Du Val point of type \type{A} or \type{D}
if $f$ is divisorial (see \ref{ge}). 
% (because
% $|-K_{Z}|$ has a Du Val member of type \type{D}, see \ref{ge}).
In both cases $H^{\n}$ has only rational singularities.
\end{scase}

For convenience of the reader we formulate the following lemma which follows from the 
standard exact sequence
\begin{equation*}
0\xrightarrow{\hspace*{20pt}} I^{(n+1)} \xrightarrow{\hspace*{20pt}} I^{(n)} \xrightarrow{\hspace*{20pt}} \gr_C^n\OOO\xrightarrow{\hspace*{20pt}} 0.
\end{equation*}

\begin{lemma}\label{lemma-grC}
Let $(X,C)$ be an extremal curve germ.
Then the following assertions hold.
\begin{enumerate}
\item \label{lemma-grC-1}
If $\Ho^1(\gr_C^n\OOO)=0$ and the 
map $\Ho^0(I^{(n)})\to \Ho^0(\gr_C^n\OOO)$ is surjective, then 
$\Ho^1(I^{(n+1)})\simeq \Ho^1(I^{(n)})$.

\item \label{lemma-grC-2}
If for all $i<n$ one has $\Ho^1(\gr_C^i\OOO)=0$ and the 
map $\Ho^0(I^{(i)})\to \Ho^0(\gr_C^i\OOO)$ is surjective, then 
$\Ho^1(I^{(n)})\simeq \Ho^1(\gr_C^n\OOO)=0$.
\item \label{lemma-grC-3}
In particular, $\Ho^1(I)= \Ho^1(\gr_C^1\OOO)=0$ and if 
$\Ho^0(\gr_C^1\OOO)=0$, then
$\Ho^1(I^{(2)})= \Ho^1(\gr_C^2\OOO)=0$.
\end{enumerate}
\end{lemma}

The following auxiliary result can be proved by induction on $n$.
\begin{proposition}\label{proposition-lP=4-XC}
Let $(X,P)\subset \CC^4_{x_1,\dots,x_4}$ be a hypersurface containing 
$C:=\{\text{$x_1$-axis}\}$ with defining equation $h\in \CC\{x_1,\dots,x_4\}$ such that
\begin{equation*}
h=x_1^mx_4+h_2(x_2,x_3)+h_3(x_1,\dots,x_4),
\end{equation*}
where $h_2$ is a quadratic form in $x_2$ and $x_3$, $h_3\in (x_2,x_3,x_4)^3$, and $m\ge 1$.
Let $I=(x_2,x_3,x_4)$ be the ideal of $C$. Let 
\begin{equation*}
\gr_C^{\bullet}:= \bigoplus_{n\ge 0} \gr_C^n\OOO
\end{equation*}
be the graded $\OOO_C$-algebra with the degree $n$ part $\gr_C^n\OOO$.
Then the following assertions hold.
\begin{enumerate}
\item \label{proposition-lP=4-XC-1}
If $h_2=0$, then
\begin{equation*}
\gr_C^2\OOO= S^2\gr_C^1\OOO.
\end{equation*}

\item \label{proposition-lP=4-XC-2}
If $h_2\neq 0$, then
\begin{equation*}
\gr_C^{\bullet}\OOO\simeq\OOO_C[x_2,x_3,x_4]/(x_1^mx_4+h_2),
\end{equation*}
where $x_2,x_3,x_4$ have degree $1$, $1$, $2$, respectively.

\item \label{proposition-lP=4-XC-4}
If $x_3^2\in h_2$, then 
\begin{equation*}
\gr_C^{\bullet}\OOO =\OOO_C[x_2,x_4]\oplus x_3\OOO_C[x_2,x_4].
\end{equation*}
\item \label{proposition-lP=4-XC-5}
If $h_2=x_2x_3$, then 
\begin{equation*}
\gr_C^{\bullet}\OOO =\OOO_C[x_4]\oplus x_2\OOO_C[x_2,x_4]\oplus x_3\OOO_C[x_3,x_4].
\end{equation*}
\end{enumerate}
\end{proposition}

\begin{proposition}\label{proposition-cases-lP}
Assume that $\Ho^0(\gr_C^1\OOO)=0$. Then 
\begin{equation}
\label{equation-gr1CO}
\gr_C^1\OOO\simeq \OOO(-1)\oplus \OOO(-1)
\end{equation}
\textup(as an abstract sheaf\textup) and one of the following possibilities holds:
\begin{enumerate}
\item
$\Sing(X)=\{P\}$, $i_P(1)=3$, and $\ell(P)=7$ or $8$,
\item
$\Sing(X)=\{P,\, R\}$, where $R$ is a type \type{(III)} point,
$i_P(1)=2$, $i_R(1)=1$, and $\ell(P)=3$ or $4$.
\end{enumerate} 
\end{proposition}

\begin{proof}
Write $\gr_C^1\OOO\simeq \OOO(a_1)\oplus \OOO(a_2)$
for some $a_1$, $a_2$.
Since $\Ho^0(\gr_C^1\OOO)=0$, we have $a_1, a_2<0$.
On the other hand, $\Ho^1(\gr_C^1\OOO)=0$ (see Lemma \ref{lemma-grC}\ref{lemma-grC-3}).
Hence, $a_1=a_2=-1$. Recall that $\ell(P)\not\equiv 2\mod 4$.

Consider the case where $P$ is the only singular point of $X$.
Then $i_P(1)=3$ by 
\cite[(2.3.2)]{Mori-1988} and \cite[(3.1.2), (4.4.3)]{Mori-Prokhorov-2008}.
According to \cite[2.16]{Mori-1988} we have $7\le \ell(P)\le 9$.
Assume that $\ell(P)=9$. Then using a deformation of the form
$\alpha_t=\alpha+t y_1y_2$ (see \eqref{equation-alpha}), we get a germ $(X_t,C_t)$
having a point $P_t$ of type \type{(IIA)} with $\ell(P_t)=1$
and two type \type{(III)} points. This is impossible by \cite[7.4.1]{Kollar-Mori-1992} and
\cite[9.1]{Mori-Prokhorov-2008}.

Suppose $\Sing(X)\neq \{P\}$.
Then by \cite[6.7]{Mori-1988} and \cite[8.6, 9.1]{Mori-Prokhorov-2008}
we have $\Sing(X)=\{P,\, R\}$, where $R$ is a type \type{(III)} point.
If $i_R(1)>1$, then by using deformation at $R$ we obtain
an extremal curve germ with one point of type \type{(IIA)}
and at least two points of type \type{(III)}.
This is impossible again by \cite[6.7]{Mori-1988} and \cite[9.1]{Mori-Prokhorov-2008}.
Therefore, $i_R(1)=1$ and so $i_P(1)=2$.
By \cite[2.16]{Mori-1988} we have $3\le \ell(P)\le 5$
Assume that $\ell(P)=5$.
Using a deformation of the form
$\alpha_t=\alpha+t y_1y_2$, we obtain a germ $(X_t,C_t)$
having a point $P_t$ with $\ell(P_t)=1$
and two type \type{(III)} points. This is impossible by \cite[7.4.1]{Kollar-Mori-1992} and
\cite[9.1]{Mori-Prokhorov-2008}.
\end{proof}

\begin{slemma}
\label{lemma-lP=3+IIIa}
If $\Ho^0(\gr_C^1\OOO)=0$, then 
\begin{equation*}
% \label{equation-remark-lP=3+IIIa}
\gr_C^2\OOO \simeq \OOO(a_1)\oplus\OOO(a_2)\oplus\OOO(a_3), 
\end{equation*}
\textup(as an abstract sheaf\textup) with $a_i\ge -1$ and $\max \{a_1,\, a_2,\, a_3\}\ge 0$.
\end{slemma}
\begin{proof}
If $\Ho^0(\gr_C^1\OOO)=0$, then the general member $H\in |\OOO_X|_C$ is singular along $C$. 
% Hence 
% the equation $\beta=0$ of $H$ at $P$ does not contain terms $y_1^ky_4$ and $y_1^ky_2$
% for any $k$.
According to \cite[Lemma 3.1.1]{Mori-Prokhorov-IIA-1}
there exists a section $\beta\in \Ho^0(I)$ containing $y_4^2$ and $y_2y_3$
at $P^\sharp$. Therefore, $\beta\in \Ho^0(I^{(2)})$ and the image 
$\bar\beta$ 
of $\beta$ in $\Ho^0(\gr_C^2\OOO)$ is non-zero.
In particular, $\Ho^0(\gr_C^2\OOO)\neq 0$. By Lemma \ref{lemma-grC}\ref{lemma-grC-3}
we have $\Ho^1(\gr_C^2\OOO)=0$ and the assertion follows.
% (see Lemma \ref{lemma-grC}\ref{lemma-grC-3})
% we have \eqref{equation-remark-lP=3+IIIa}.
% \begin{equation}\label{equation-remark-lP=3+IIIa}
% \gr_C^2\OOO \simeq \OOO(a_1)\oplus\OOO(a_2)\oplus\OOO(a_3), 
% \end{equation}
% (as an abstract sheaf) with $a_i\ge -1$ and $\max \{a_1,\, a_2,\, a_3\}\ge 0$.
% % this implies $\chi(\gr_C^2\OOO)\ge 1$
% % and $\deg \gr_C^2\OOO\ge -2$ by Riemann-Roch since $\gr_C^2\OOO$ is of rank $3$.
\end{proof}

\section{Cases $\ell(P)=3$ and $7$}
\label{section-lP=3-III}
In this section we assume that $\ell(P)\in \{3,\, 7\}$.
It will be shown that Computation \ref{computation-lP=3a-III} is applicable here
and the possibility \ref{main-theorem-divisorial} occurs.

\begin{case}\label{notation-lP=3+III}
By Proposition \ref{proposition-cases-lP} in the case $\ell(P)=3$ the variety
$X$ has a type \type{(III)} point $R$ with $i_R(1)=1$ and 
$X$ is smooth outside $P$ in the case $\ell(P)=7$.
According to \ref{equation-IIA-point} the equation of $X$ at $P$ has the form
\begin{equation}\label{equation-alpha-lP=3-and-7}
\alpha=y_1^{\ell(P)}y_3+y_2^2+y_3^2+\delta y_4^{2k+1}+c y_1^2y_4^2+\epsilon y_1y_3y_4+\xi y_1^3y_2y_4
+\cdots=0.
\end{equation}
Thus
\begin{equation}\label{equation-alpha-lP=3-and-7-mod}
\alpha\equiv y_1^{\ell(P)}y_3+y_2^2
\mod (y_2y_4,\,  y_4^2)+ I^{(3)},\qquad y_3\in I^{(2)}.
\end{equation}
\end{case}

\begin{scase}
In the case $\ell(P)=3$
by \cite[Lemma 2.16]{Mori-1988}, since $i_R(1)=1$,
the equation of $X$ at $R$ has the form
\begin{equation}\label{equation-beta-lP=3+III-1}
\gamma=z_1z_3+\gamma_2(z_2,z_4)+\gamma_3(z_1,\dots,z_4),
\end{equation}
where $\gamma_2$ is a quadratic form, $\gamma_3\in (z_2,z_4)^3+(z_2,z_4)z_3+(z_3)^2$,
and $C$ is the $z_1$-axis.
\end{scase}
\begin{scase}
According to \eqref{equation-gr1CO}, since
$y_4$ and $y_2$ form an $\ell$-free $\ell$-basis of $\gr^1_C\OOO$ at $P$, we have the following $\ell$-isomorphism
\begin{equation}
\label{equation-(7.4.1.1)-lP=3+III}
\vcenter{
\xymatrix@R=6pt@C=-3pt{
\gr_C^1\OOO= &(-1+3P^\sharp)\ar@{=}[d]&\toplus& (-1+2P^\sharp).\ar@{=}[d]
\\
& \AAA && \BBB
}}
\end{equation}
We choose the coordinates
$y_1,\dots, y_4$ at $P$ keeping $y_1$ and $y_3$ the same
so that
$y_2$ is an $\ell$-basis of $\AAA$
and
$y_4$ is an $\ell$-basis of $\BBB$.
\end{scase}

\begin{sremark}\label{remarkProposition-lP=3+IIIa}
By \eqref{equation-alpha-lP=3-and-7-mod} the semi-invariants
$y_4^2$, $y_2y_4$, $y_3$ form an $\ell$-basis of $\gr_C^2\OOO$.
\end{sremark}

\begin{lemma}\label{lemma-possibilities-lP=3+III}
{}For $\gr_C^2\OOO$, one of the following possibilities holds
\begin{numcases}{\gr_C^2\OOO=}
(a)\toplus (-1+P^\sharp)^{\toplus 2},\quad a=0,\ 1 
\label{equation-possibilities-lP=3+III-a}
\\
(P^\sharp)\toplus (0)\toplus (-1+P^\sharp), 
\label{equation-possibilities-lP=3+III-b}
\\
\VVV\toplus (-1), 
\label{equation-possibilities-lP=3+III-c}
\end{numcases}
where $\VVV$ is some $\ell$-sheaf.
\end{lemma}

\begin{proof}
Consider the natural map
\begin{equation}
\label{equation-varphi}
\varphi: \tilde S^2 \gr_C^1\OOO=
\AAA^{\totimes 2} 
\toplus (\AAA\totimes\BBB)\toplus
\BBB^{\totimes 2}
\xrightarrow{\hspace*{25pt}}\gr_C^2\OOO,
\end{equation}
where
\begin{equation*}
\AAA^{\totimes 2}= (-1+2P^\sharp),\quad
\AAA\totimes\BBB= (-1+P^\sharp),\quad
\BBB^{\totimes 2}= (-1).
\end{equation*}
$\ell$-bases of these $\ell$-sheaves at $P$ are
$y_2^2$,\, 
$y_2y_4$,\, 
$y_4^2$, and  respectively.
By Remark \ref{remarkProposition-lP=3+IIIa}
we see that
an $\ell$-basis of $\gr_C^2\OOO$ can be taken as
$y_4^2,\, y_2y_4,\, y_3$.
% Since $y_1^{\ell(P)}y_3+y_2^2\equiv 0 \mod (y_2y_4,\, y_4^2)+I^{(3)}$,
According to \eqref{equation-alpha-lP=3-and-7-mod}
we have $y_1^2y_2^2\equiv (\unit)\cdot y_1^{\ell(P)+1}\cdot y_1y_3$. 
Hence, 
\begin{equation}
\label{equation-lP=3-7-Coker-P}
\coker_P \varphi=\OOO_C \cdot y_1y_3/\OOO_C\cdot y_1^2y_2^2=
\OOO_C/\bigl(y_1^{\ell(P)+1}\bigr)\cdot y_1y_3
\end{equation}
and $\coker_P \varphi^{\sharp}= \OOO_{C^{\sharp}}/\bigl(y_1^{\ell(P)}\bigr)\cdot y_3$.
In particular, 
$\len _P\coker_P \varphi=(\ell(P)+1)/2$.
If $\ell(P)=3$, then 
\begin{equation*}
\coker_R \varphi =
\begin{cases}
\OOO_C/(z_1)\cdot z_3 & \text{if $\gamma_2\neq 0$,} 
\\
0 & \text{if $\gamma_2= 0$,} 
\end{cases}
\end{equation*}
In particular,
$\len_R \coker_R \varphi\le 1$.
By  Lemma \ref{lemma-lP=3+IIIa} one of the following holds
\begin{equation*}
\gr_C^2\OOO\simeq
\OOO(-1)^{\oplus 2}\oplus\OOO(1),\quad
\OOO(-1)^{\oplus 2}\oplus\OOO,\quad\text{or}\quad
\OOO^{\oplus 2}\oplus\OOO(-1).
\end{equation*}
By Remark \ref{remarkProposition-lP=3+IIIa} we get the only possibilities listed in
Lemma \xref{lemma-possibilities-lP=3+III}.
\end{proof}

\begin{lemma}\label{treating-equation-possibilities-lP=3+III-c}
The case \eqref{equation-possibilities-lP=3+III-c} does not occur.
\end{lemma}

\begin{proof}
Indeed, from the exact sequence
\begin{equation*}
0\xrightarrow{\hspace*{20pt}} \gr_C^1\omega\xrightarrow{\hspace*{20pt}} \omega /F^{2}\omega
\xrightarrow{\hspace*{20pt}} \omega/ F^1 \omega\xrightarrow{\hspace*{20pt}} 0,
\end{equation*}
we obtain $\chi(\omega /F^{2}\omega)=0$. Then we apply
\cite[Lemma 3.7(ii)]{Mori-Prokhorov-IIA-1} with $\KKK=I^{(2)}$.
\end{proof}

\begin{lemma}
\label{lemma-equation-possibilities-lP=3+IIIa}
The case \eqref{equation-possibilities-lP=3+III-a} does not occur.
\end{lemma}

\begin{proof}
The deformation of the form
\begin{equation}
\label{equation-lP=3-7-deformations}
\alpha'=\alpha+\delta' y_4^{3}+\epsilon' y_1y_3y_4
\end{equation}
does not change the case division of
Lemma \xref{lemma-possibilities-lP=3+III} because $y_4^3,\, y_1y_3y_4\in I^{(3)}$.
Since it suffices to disprove a small deformation of $X$, we may assume that in \eqref{equation-alpha-lP=3-and-7}
the coefficients $\delta$ and $\epsilon$ are general and $k=1$.

Let us analyze  the map $\varphi$ (see \eqref{equation-varphi}) in our case.
% We have
% \begin{equation*}
% \gr^1_C\OOO= (-1+2P^\sharp)\toplus(-1+3P^\sharp), \quad
% \gr^2\OOO=
% (a)\toplus (-1+P^\sharp)^{\toplus 2}.
% \end{equation*}
% By our assumptions $y_2^2$ is an $\ell$-basis of
% $\AAA^{\totimes 2}$. 
% On the other hand,
Since the map $\AAA^{\totimes 2}\to (-1+P^\sharp)$ is zero (by the degree consideration),
the image of
$\AAA^{\totimes 2}=(-1+2P^\sharp) \hookrightarrow \gr_C^2\OOO$
must be contained in the first summand $(a)\subset \gr_C^2\OOO$. Since $(-1+P^\sharp)^{\toplus 2}$ has
no global sections, $\beta$ must be a global section of $(a)$.
The map $\varphi$ is given by the following matrix:
\[
\arraycolsep=1.4pt
\begin{blockarray}{cc@{\qquad}ccc}
&&\scriptstyle{(-1+2P^\sharp)}&\scriptstyle{(-1+P^\sharp)}&\scriptstyle{(-1)}
\\[7pt]
\begin{block}{rc@{\qquad}(ccc)}
\scriptstyle{(a)}&\scriptstyle{v_1}&y_1^2h(y_1^4)&\star&\star\star
\\
\scriptstyle{(-1+P^\sharp)}&\scriptstyle{v_2}&0&b_1&b_3y_1
\\
\scriptstyle{(-1+P^\sharp)}&\scriptstyle{v_3}&0&b_2&b_4y_1
\\
\end{block}
\end{blockarray}
\]
where $b_1,\dots, b_4$ are constants
and $h$ is a polynomial of degree $\le a$. Since the matrix is non-degenerate, $(b_1b_4-b_2b_3)h\neq 0$.
Applying elementary transformations  of rows
and switching the second and the third rows
(which correspond to automorphisms of 
$\gr_C^2\OOO$), one can reduce the matrix to the form 
\begin{equation}
\label{equation-lP=3-7-matrix}
\begin{pmatrix}
y_1^2h(y_1^4)&0&b_5
\\
0&1&0
\\
0&0&y_1
\end{pmatrix}
\end{equation}
where $b_5$ is a constant.
If $b_5=0$, then 
\[
(\coker_{P} \varphi)^\sharp \simeq 
\OOO_C^\sharp /(y_1)\oplus \OOO_C^\sharp /(y_1^2h).
\]
This contradicts \eqref{equation-lP=3-7-Coker-P}. Hence, we may assume that $b_5=1$.
From the matrix \eqref{equation-lP=3-7-matrix}  we see
\begin{eqnarray*}
y_2^2&=& y_1^2hv_1,
\\
y_4^2&=& v_1+y_1v_3.
\end{eqnarray*}
Eliminating $v_1$ we obtain the following relations in $\gr_C^2\OOO$:
\begin{equation}
\label{equation-lP=3-7-vv}
v_1=y_4^2-y_1v_3,\qquad 
y_1^3h v_3+y_2^2-y_1^2hy_4^2=0.
\end{equation}
The last one must a multiple of $\alpha$.
\begin{scase}
\label{scase-lP=3-7-new-treatment}
If $h$ is a unit, then  
comparing with  \eqref {equation-alpha-lP=3-and-7}  we see that $\ell(P)=3$, 
$c=h(0)\neq 0$ and $v_3\ni y_3$.
If $h$ is linear, then  $\ell(P)=7$, 
$c=h(0)= 0$ and again $v_3\ni y_3$.
Since $\beta$ is a section of $(a)\subset \gr_C^2\OOO$, it must be proportional 
to $v_1$. Therefore, $y_1y_3\in \beta$. Moreover,  \eqref{equation-lP=3-7-vv}
shows that in the case $\ell(P)=3$ the term $y_1y_3$ appears in $\beta$ with coefficient $1/c$.
Note that the coefficients of $y_1^2y_4^2\in \alpha$ and $y_4^2,\, y_1y_3\in \beta$
are preserved under deformations \eqref{equation-lP=3-7-deformations}.
So we may assume that the condition $\epsilon c\neq \delta$
of \ref{computation-lP=3+III-part2} is satisfied.
Thus in the case $\ell(P)=3$ we may apply Computation \ref{computation-lP=3+III-part2}.
In the case $\ell(P)=7$ we also  may apply \ref{computation-lP=3+III-part2}
to $\alpha^o=\beta=0$, where $\alpha^o$ is a linear combination of 
$\alpha$ and $y_1^2\beta$  (and so $y_1^2y_4^2\in \alpha^o$).
Then in both cases we obtain a contradiction by Lemma \ref{slemma-lP=3+III-generalityH} below.
\end{scase}

\begin{slemma}\label{lemma-computation-lP=3+III-part2}
Assume that $\Delta(H,C)$ at $P$ is as in
\eqref{graphs-computation-lP=3+III-part2}.
Then the contraction $f$ is birational and
$\Delta(H,C)$
has one of the following forms:
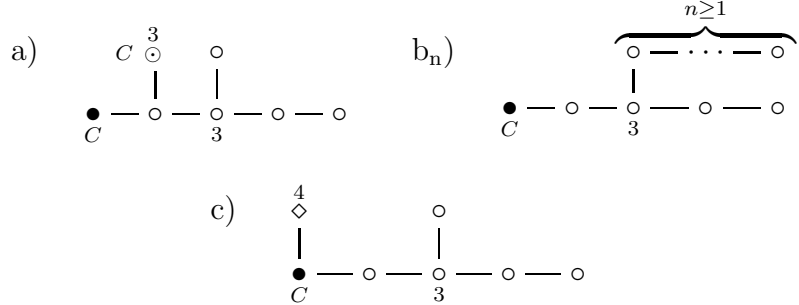
\begin{equation*}
\xy
\xymatrix@R=1pt@C=10pt{
\mathrm{a)}&&\hbox to 5pt{\hss {$\scriptstyle{C}$ $\overset{3}{\scriptstyle \odot}$}}\ar@{-}[d]&\circ\ar@{-}[d]
\\
&\underset C\bullet\ar@{-}[r] &\circ\ar@{-}[r] &\underset{3}\circ\ar@{-}[r]&\circ \ar@{-}[r]&\circ
}
\endxy
\hspace{17pt}
\xy
\xymatrix"M"@C=10pt@R=1pt{
\mathrm {b_n)}&&&\circ\ar@{-}[d]\ar@{-}[r]& \cdots\ar@{-}[r]&\circ
\\
&\underset {C}{\bullet}\ar@{-}[r] &\circ\ar@{-}[r]
&\underset{3}{\circ}\ar@{-}[r]&\circ\ar@{-}[r]&\circ
}
\POS"M1,4"."M1,6"!C*\frm{^\}},+U*++!D\txt{$\scriptstyle{n\ge 1}$}
\endxy
\end{equation*}
\begin{equation*}
\xy
\xymatrix@R1pt@C13pt{
\mathrm {c)}&\overset{4}{\diamond}\ar@{-}[d]&&\circ\ar@{-}[d]
\\
&\underset{C}\bullet\ar@{-}[r]&\circ\ar@{-}[r] &\underset{3}\circ\ar@{-}[r] &\circ\ar@{-}[r] &\circ
}
\endxy
\end{equation*}
where $\bullet$, as usual, corresponds to a component of the proper transform of $C$ 
that is a $(-1)$-curve,
$\scriptstyle \odot$ corresponds to a component that is not a $(-1)$-curve,
and $\diamond$ corresponds to an exceptional divisor over a point on $C\setminus \{P\}$.
\end{slemma}

\begin{proof}
Let $H^{\n}\to \tilde H$ be the normalization, let 
$\hat H\to H^{\n}$ be the minimal resolution,
% 
% $\hat H\to H^{\n}\to \tilde H$ be
% the composition of the normalization and the minimal resolution
and let $\hat C\subset \hat H$ be
the proper transform of $C$.
Assume that $\hat C$ has two components $\hat C_1$ and $\hat C_2$
(the case \eqref{graphs-computation-lP=3+III-part2}a)).
Then $\Delta(H,C)$ has the form
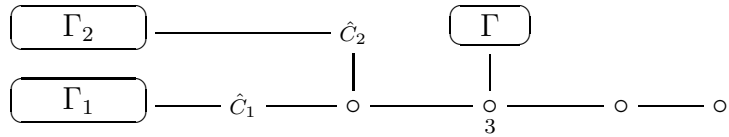
\begin{equation*}
\xymatrix@R1pt {
\ovalh{\phantom{PP}$\Gamma_2$\phantom{PP}}
&&\ar@{-}[ll]\scriptstyle{\hat C_2}\ar@{-}[d]&
\ovalh{\phantom{P}$\Gamma$\phantom{P}}
\\
\ovalh{\phantom{PP}$\Gamma_1$\phantom{PP}}\ar@{-}[r]
&\scriptstyle{\hat C_1}\ar@{-}[r]&
\circ\ar@{-}[r]&\underset3\circ\ar@{-}[r]\ar@{-}[u]&\circ\ar@{-}[r]&\circ
}
\end{equation*}
where subgraphs $\Gamma_1$ and $\Gamma_2$ correspond to singularities  
of $H^{\n}$ outside $P$ and $\Gamma$ is a Du Val
subgraph corresponding to 
$O'\in \tilde H$ (see \ref{claim-new-11-3}).
Since the whole configuration $\Delta(H,C)$ is contractible to
a Du Val point
or corresponds to a fiber of a rational curve fibration (see \ref{sde1}), it contains a $(-1)$-curve.
Thus we may assume by symmetry that $\hat C_1^2=-1$. Then contracting $\hat C_1$ we obtain
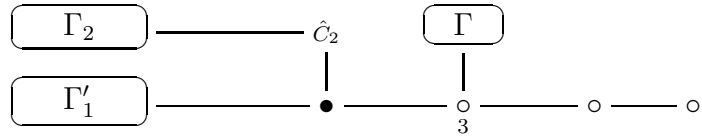
\begin{equation*}
\xymatrix@R1pt {
\ovalh{\phantom{PP}$\Gamma_2$\phantom{PP}}&&\ar@{-}[ll]\scriptstyle{\hat C_2}\ar@{-}[d]
&\ovalh{\phantom{P}$\Gamma$\phantom{P}}
\\
\ovalh{\phantom{PP}$\Gamma_1'$\phantom{PP}}\ar@{-}[rr]&&\bullet\ar@{-}[r]&
\underset3\circ\ar@{-}[r]\ar@{-}[u]&\circ\ar@{-}[r]&\circ
}
\end{equation*}
Then $\Gamma_1'$ must be empty. Contracting the black vertex we obtain
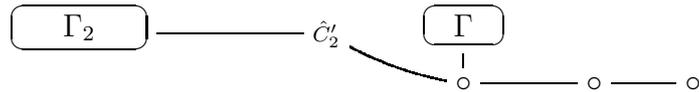
\begin{equation*}
\xymatrix@R1pt {
\ovalh{\phantom{PP}$\Gamma_2$\phantom{PP}}&&\ar@{-}[ll]\scriptstyle{\hat C_2'}\ar@{-}@/_3pt/[dr]
&\ovalh{\phantom{P}$\Gamma$\phantom{P}}
\\
&&&\circ\ar@{-}[r]\ar@{-}[u]&\circ\ar@{-}[r]&\circ
}
\end{equation*}
Recall that $\Gamma\neq \emptyset$. 
It is easy to see that configuration $\Delta(H,C)$ does not
correspond to a fiber of a rational curve fibration. Hence $f$ is birational.
Since $y_4^3\in \alpha$, the general member $D\in |-K_X|$ is of type \type{D_5}
(see \ref{ge}). Hence $f(H)$ is either of type \type{D_5} or ``better''.
This implies that 
$\Gamma_2=\emptyset$, $\hat C_2'^2=-1$ and so $\hat C_2^2=-2$. Moreover, $\Gamma$ consists of a single vertex.
Thus we obtain the case a).
The cases where $\hat C$ is irreducible is treated in a similar way.
\end{proof}

\begin{slemma}\label{slemma-lP=3+III-generalityH}
Assume that $(H,C)$ is of type
\type{a)}, \type{b_n)} or \type{c)}
of Lemma \xref{lemma-computation-lP=3+III-part2}.
Then the chosen element $H$
is not general in $|\OOO_X|_C$.
\end{slemma}
\begin{proof}
Take a divisor $\Theta$ on the minimal resolution
whose coefficients for \type{a)} and \type{b_n)} are as follows:
\begin{equation*}
\xy
\xymatrix@R=0pt@C=10pt{
\mathrm{a)}&\overset{1}{\scriptstyle \odot}\ar@{-}[d]&\overset{1}\circ\ar@{-}[d]&&\overset{1}\vartriangle\ar@{-}[d]
\\
\underset 3\bullet\ar@{-}[r] &\underset 3\circ\ar@{-}[r] &
\underset{2}\circ\ar@{-}[r]&\underset 2\circ \ar@{-}[r]&\underset 2\circ&\underset 1\vartriangle\ar@{-}[l]
}
\endxy
\hspace{25pt}
\xy
\xymatrix@C=17pt@R=1pt{
\mathrm {b_n)}&&\overset{1}\circ\ar@{-}[d]\ar@{-}[r]& \cdots\ar@{-}[r]&\overset{1}\circ&\overset{1}\vartriangle\ar@{-}[l]
\\
\underset {1}{\bullet}\ar@{-}[r] &\underset{1}\circ\ar@{-}[r]
&\underset{1}{\circ}\ar@{-}[r]&\underset{1}\circ\ar@{-}[r]&\underset{1}\circ&\underset{1}\vartriangle\ar@{-}[l]
}
\endxy
\end{equation*}
where $\vartriangle$ corresponds to
an arbitrary smooth analytic curve meeting the corresponding component
transversely.
It is easy to verify that $\Theta$ is numerically trivial,
so $\Theta$ is the pull-back of a Cartier divisor $\Theta_Z$ on $H_Z$.
Clearly, $\Theta_Z$ extends to a Cartier divisor $G_Z$ on $Z$.
Let $G:=f^*G_Z$. Then $\Theta$ is the pull-back of $G|_H$.

In the case \type{a)} the normalization of $H$ at a general point of $C$
is locally reducible: $H=H_1+H_2$.
The diagram \type{a)} shows that for $H_i\cap G$ is a reduced divisor
for some $i\in \{1,\, 2\}$.
Hence, $G\in |\OOO_X|_C$ is normal which contradicts our assumptions.
In the case \type{b_n)} and \type{c)}
the normalization of $H$ is a bijection by
Corollary \xref{scorollary-lP=3+III-sing}.
In the case \type{b_n)} 
it is easy to see that
the multiplicity of the intersection $H\cap G$ at a general point of $C$
is $\le 2$. This shows that the divisor $G\in |\OOO_X|_C$ is normal, a contradiction.

Similar arguments show that in the case \type{c)} the multiplicity of the intersection
$H\cap G$ at a general point of $C$
equals $4$.
By Corollary \ref{scorollary-lP=3+III-sing}
$H$ has a cuspidal singularity at a general point of $C$.
Let $D\subset X$ be a disk that intersects $C$ transversely at a general point.
Then the curves $H|_D$ and $G|_D$ are cuspidal.
Since $H|_D \cdot G|_D=H\cdot G\cdot D=4$, these cusps are in general position,
that is, the quadratic parts of the corresponding equations are not proportional.
But then the general member of the pencil generated by $H|_D$ and $G|_D$
has an ordinary double point at the origin.
Hence the chosen element $H\in |\OOO_X|_C$ is not general, a contradiction.
\end{proof}
% 
% Thus the case \eqref{equation-possibilities-lP=3+III-a} with $\ell(P)=3$ does not occur.
% 
% \begin{scase}\label{claim-D-6+k}
% It remains to consider the subcase where $h_2$ is not a unit (and then $\ell(P)=7$).
% By \eqref{equation-ell-P=3-hh} \ $h_1$ must be a unit and we have $y_1^5y_3\in \beta$.
% Hence, we can apply Computation \ref{computation-lP=3+III-part3}
% with $\epsilon\neq0$.
% Since the configuration \eqref{graphs--lP=3+III-part2-new} is not contractible to
% neither a smooth point nor a Du Val point of type \type{A} or \type{D},
% and is not a fiber of
% a rational curve fibration (see \ref{sde1}), we get a contradiction.
% \end{scase}
% 
% Thus neither $h_1$ nor $h_2$ is a unit. On the other hand, 
% by \eqref{equation-ell-P=3-hh} either $h_1$ or $h_2$ must be a unit.
Thus the case \eqref{equation-possibilities-lP=3+III-a} does not occur.
Lemma \ref{lemma-equation-possibilities-lP=3+IIIa} is proved.
\end{proof}

\begin{case}{\bf Case \eqref{equation-possibilities-lP=3+III-b}.}
\label{Subcase-equation-possibilities-lP=3+III-c}
We will show that Computation \ref{computation-lP=3a-III} is applicable in this case and 
the possibility \ref{main-theorem-divisorial} occurs. 
We have
\begin{equation*}
% \label{equation-(7.4.1.1)-lP=3+III}
\vcenter{
\xymatrix@R=6pt@C=-3pt{
\gr_C^2\OOO= &(P^\sharp)\ar@{=}[d]&\toplus& (0)\ar@{=}[d]&\toplus& (-1+P^\sharp).\ar@{=}[d]
\\
& \DDD && \EEE && \GGG
}}
\end{equation*}

We apply the arguments similar to  that used in 
the proof of Lemma \ref{lemma-equation-possibilities-lP=3+IIIa}.
In our case the map $\varphi$ is given by the following matrix:
\[
\begin{blockarray}{cc@{\qquad}ccc}
&&\scriptstyle{(-1+2P^\sharp)}&\scriptstyle{(-1+P^\sharp)}&\scriptstyle{(-1)}
\\[7pt]
\begin{block}{rc@{\qquad}(ccc)}
\scriptstyle{(P^\sharp)}&\scriptstyle{w_1} &b_1y_1^3&h(y_1^4)&\star
\\
\scriptstyle{(0)} &\scriptstyle{w_2}    &b_2y_1^2&b_3y_1^3&\star\star
\\
\scriptstyle{(-1+P^\sharp)}& \scriptstyle{w_3}   &0&b_4&b_5y_1
\\
\end{block}
\end{blockarray}
\]
where $b_1,\dots,b_5$ are constants, $h$ is  a polynomial of degree $\le 1$, and $\star$
is divisible by $y_1$. 
Consider the map
\begin{equation}
\label{equation-lP=3-7-definition-pi}
\pi: (-1+P^\sharp)=\AAA\totimes\BBB \xrightarrow{\hspace*{10pt}}
\gr^2_C\OOO
\xrightarrow{\makebox[20pt]{ $\scriptstyle\pr$}} 
\GGG=(-1+P^\sharp),
\end{equation}
which is uniquely determined by $\AAA$ and $\BBB$.
We may regard $\pi$ as the multiplication by  $b_4$.

\begin{slemma}
$b_4\neq 0$.
\end{slemma}
\begin{proof}
Assume that  $b_4=0$. Since the matrix is non-degenerate, $b_5\neq0$.
Applying elementary transformations  of rows, as in the proof of 
Lemma \ref{lemma-equation-possibilities-lP=3+IIIa}, one can reduce the matrix to the form 
\begin{equation}
\label{equation-lP=3-7-matrix-2}
\begin{pmatrix}
b_1y_1^3&h(y_1^4)&0
\\
b_2y_1^2&b_3y_1^3&b_6
\\
0&0&y_1
\end{pmatrix}
\end{equation}
where $b_6$ is a constant.
If $b_6=0$, then 
\[
(\coker_{P} \varphi)^\sharp \simeq 
\OOO_C^\sharp /(y_1)\oplus \text{(non-zero $\OOO_C^\sharp$-module)}.
\]
This contradicts \eqref{equation-lP=3-7-Coker-P}. Hence, we may assume that $b_6=1$.
Assume that $b_2=0$.
% and $h$ is a unit. 
Applying elementary row transformations 
we can reduce \eqref{equation-lP=3-7-matrix-2} to the form
\[
 \begin{pmatrix}
y_1^3&h(y_1^4)&0
\\
0&b_3y_1^3&1
\\
0&0&y_1
\end{pmatrix}
\]
which gives us
\begin{eqnarray*}
y_2^2&=& y_1^3 w_1,
\\
y_2y_4&=& h(y_1^4) w_1+b_3y_1^3 w_2,
\\
y_4^2&=& w_2+y_1w_3.
\end{eqnarray*}
If $h(0)=0$, then  one can see that 
$(\coker_{P} \varphi)^\sharp$ cannot be a cyclic $\OOO_C^\sharp$-module.
Thus, $h$ is a unit and we can eliminate $w_1$ and $w_2$:
\begin{eqnarray*}
y_2^2&=&     \textstyle{\frac 1h y_1^3y_2y_4-\frac {b_3}h y_1^6 y_4^2+\frac {b_3}h y_1^7w_3},
\\
 w_1&=& \textstyle{\frac 1h y_2y_4-\frac {b_3}h y_1^3 y_4^2+\frac {b_3}h y_1^4w_3},
\\
 w_2&=&y_4^2-y_1w_3.
\end{eqnarray*}
Comparing the first equation with \eqref{equation-alpha-lP=3-and-7}
we see that $\ell(P)=7$ and $w_3\ni y_3$.
Then from the second one we see $w_1\not\ni y_3$. 
Clearly, $\beta$ is a linear combination of $y_1w_1$ and $w_2$
(with constant coefficients). Hence, $\beta\ni y_1y_3$.
As in the proof of Lemma \ref{lemma-equation-possibilities-lP=3+IIIa}
a deformation of the form \eqref{equation-lP=3-7-deformations} is trivial 
modulo $I^{(3)}$ and so it
preserves case division \ref{lemma-possibilities-lP=3+III}, as well as, 
the vanishing of $b_4$.
Then we can argue as in \ref{scase-lP=3-7-new-treatment} and get a contradiction.

Hence $b_2\neq 0$. Then we may assume that $b_1=0$ and  $b_2=1$. The relations in 
$(\coker_{P} \varphi)^\sharp$ are $y_1^2w_2=w_2+y_1w_3=0$, $hw_1+b_3y_1^3w_2=0$.
Eliminating $w_2$ one can see
\[
 (\coker_{P} \varphi)^\sharp\simeq 
\OOO_C^\sharp /(y_1^3)\oplus \OOO_C^\sharp /(h).
\]
By \eqref{equation-lP=3-7-Coker-P}
we have $h(0)\neq 0$ and $\ell(P)=3$.
From the matrix \eqref{equation-lP=3-7-matrix-2}  we see
\begin{eqnarray*}
y_4^2&=& w_2+y_1w_3,
\\
y_2^2&=& y_1^2w_2,
\\
y_2y_4&=& h(y_1^4)w_1+ b_3y_1w_2.
\end{eqnarray*}
Eliminating $w_2$ we obtain the following relations in $\gr_C^2\OOO$:
\begin{equation}
\label{equation-lP=3-7-vv-2}
w_2=y_4^2-y_1w_3,\qquad
y_2^2- y_1^2y_4^2+y_1^3w_3=0.
% y_2y_4= lw_1+ b_3y_1y_4^2- b_3y_1^2w_3.
\end{equation}
The last must be congruent to $ \alpha\mod I^{(3)}$. 
Comparing with  \eqref {equation-alpha-lP=3-and-7}  we see that $w_3=\frac 1c y_3$ in $\gr_C^2\OOO$.
Since $\beta$ is a section of $(0)\subset \gr_C^2\OOO$, it must be proportional 
to $w_2$. Therefore, $y_1y_3\in \beta$. Moreover, \eqref{equation-lP=3-7-vv-2}
shows that $y_1y_3$ appears in $\beta$ with coefficient $1/c$.
Now we apply Computation \ref{computation-lP=3+III-part2},
Lemma \ref{lemma-computation-lP=3+III-part2},
and Lemma \ref{slemma-lP=3+III-generalityH}
and get a contradiction.
\end{proof}

\begin{scase}
\label{treating-6-5-5}
From now on we assume that $b_4\neq 0$.
In other words, the map $\pi$
is non-zero.
The induced map 
\[
\BBB^{\totimes 2}=(-1)\longrightarrow \GGG=(-1+P^\sharp)
\]
can be regarded as the multiplication by $sy_1$ for some $s$.
For $\mu\in \CC$, take a subsheaf $\BBB'\subset \AAA\toplus\BBB$ so that
$y_4':=y_4+\mu y_1y_2$ is an $\ell$-basis of $\BBB'$.
Clearly, $\gr^1_C\OOO =\AAA\toplus\BBB'$.
Regard $y_1$ as a map $\BBB\to \AAA$.
Then $(\mu y_1,1)(\BBB)\subset \AAA\toplus \BBB$ and
we have the following diagram
\begin{equation*}
\xymatrix@R10pt{
((\mu y_1,1)(\BBB))^{\totimes 2}\ar@{=}[d]\ar@{^{(}->}[r]& \tilde S^2\gr_C^1\OOO\ar[r]^-{\pr}&\GGG
\\
(\mu^2 y_1^2,2\mu y_1,1)(\BBB^{\totimes 2})\ar@/_15pt/[urr]_-{\cdot(2\mu y_1b_4+sy_1)}
}
\end{equation*}
Set $\mu:=-s/(2b_4)$. With this choice of $\mu$, the map
${\BBB'}^{\totimes 2}\to \GGG$ is zero.
Thus
$\AAA^{\totimes 2}\toplus\BBB'^{\totimes 2}\subset \DDD \toplus\EEE$.
Let $\KKK$ be the ideal such that $I^{(2)}\supset \KKK\supset I^{(3)}$
and $\KKK/ I^{(3)}= \DDD\toplus \EEE$.

Since $\AAA^{\totimes 2}\to \GGG$ is zero, perturbing $\BBB$
with $\mu$ has no effect on $\pi: \AAA \totimes \BBB \to \GGG$, and we use the
same notation $\pi: \AAA \totimes \BBB' \to \GGG$.
\end{scase}

\begin{slemma}\label{lemma-lP=3-7-ci}
$I\KKK=I^{(3)}$ outside $P$ and
$I^{\sharp}\KKK^{\sharp}=(I^{(3)})^{\sharp}$ at $P$.
\end{slemma}

\begin{proof}
Consider the following digram with $\ell$-exact rows and injective vertical arrows:
\begin{equation}\label{big-diagram}
\vcenter{
\xy
\xymatrix@R=14pt@C=20pt{
0\ar[r]
&\AAA^{\totimes 2}\toplus \BBB'^{\totimes 2}\ar[r]\ar@{^{(}->}[d]^{\upsilon}
&\tilde S^2\gr^1_C\OOO\ar[r]\ar@{^{(}->}[d]^{\varphi}&
\AAA\totimes \BBB'\ar[r]\ar[d]_{{\simeq}}^{b_4}& 0
\\
0\ar[r]&
\DDD\toplus \EEE\ar[r]
&\gr_C^2\OOO\ar[r]&
\GGG\ar[r] &0
}
\endxy
}
\end{equation}
At a point $Q\in C$ which is a smooth point of $X$,
we can choose coordinates $u_1,u_2,u_3$
for $(X,Q)$
so that $Q$ is the origin, $C$ is the $u_1$-axis,
and $u_2$ (resp. $u_3$) generates $\AAA$ (resp. $\BBB'$)
at $Q$. Then from \eqref{big-diagram} we see
\begin{equation*}
I^{(3)}=I^3=(u_2,u_3)^3, \qquad
\KKK=(u_2^2,u_3^2)+(u_2,u_3)^3,
\end{equation*}
from which follows $I^{(3)}=\KKK I$.
At $P$, again from \eqref{big-diagram} we have
\begin{equation*}
\coker_{P^\sharp}\upsilon^\sharp \simeq \coker_{P^\sharp} \varphi^\sharp
\simeq \left(\OOO_{C^\sharp}/ (y_1^3)\right) y_3.
\end{equation*}
Thus, $(\DDD\toplus \EEE)^\sharp$ is generated by $y_3$ and $\varrho$, where
$\varrho:=y_2^2$ or $y_4'^2$. Therefore,
\begin{eqnarray*}
y_2^2,\ y_4'^2\in \KKK^\sharp&=&(y_3,\varrho)+(y_2, y_4)^3,
\\
\KKK^\sharp I^\sharp &=& y_3 I^\sharp+(y_2, y_4)^3.
\end{eqnarray*}
Whence,
\begin{equation*}
\OOO_{C^\sharp}\cdot y_3 \oplus \OOO_{C^\sharp}\cdot \varrho
\twoheadrightarrow \KKK^\sharp /\KKK^\sharp I^\sharp.
\end{equation*}
Since
\begin{equation*}
\KKK^\sharp /\KKK^\sharp I^\sharp \twoheadrightarrow
\KKK^\sharp/ {I^{(3)}}^\sharp
\simeq \OOO_{C^\sharp}\oplus \OOO_{C^\sharp},
\end{equation*}
the arrow above is an isomorphism and
$I^{\sharp}\KKK^{\sharp}=(I^{(3)})^{\sharp}$ at $P^{\sharp}$.

If $\ell(P)=3$, then at $R$, changing coordinates $z_1,\dots, z_4$ keeping $z_1$ and $z_3$
the same, we may assume that $z_2$ and $z_4$
are bases at $R$ of $\AAA$ and $\BBB'$, respectively.
Then in view of \eqref{big-diagram} and $\coker_R \varphi =\CC_R$, we see that
$\DDD\toplus \EEE$ is generated by $z_3$ and $z_i^2$ for some $i=2,\, 4$.
Therefore,
\begin{eqnarray*}
z_2^2,\, z_4^2\in \KKK &=&(z_3,\, z_i^2)= (z_2,\, z_4)^3,
\\
\KKK I &=& z_3 I+(z_2, y_4)^3.
\end{eqnarray*}
Whence,
\begin{equation*}
\OOO_{C} \cdot z_3 \oplus \OOO_{C}\cdot z_i^2
\twoheadrightarrow \KKK /\KKK I.
\end{equation*}
Since
\begin{equation*}
\KKK /\KKK I \twoheadrightarrow \KKK/ I^{(3)}
\simeq \OOO_C\oplus \OOO_C,
\end{equation*}
we have $I\KKK=I^{(3)}$ at $R$.
This proves Lemma \ref{lemma-lP=3-7-ci}.
\end{proof}

\begin{scorollary}\label{corollary-lP=3-7-ci}.
$\KKK\totimes \OOO_C\simeq (P^\sharp)\toplus (0)$ and so $\KKK$ is an l.c.i.
ideal of codimension $2$ outside $P$ and $\KKK^\sharp$ is l.c.i. at $P^\sharp$.
\end{scorollary}

\begin{scase}
Thus,
\begin{eqnarray*}
\KKK/ (\KKK\totimes I)&=& (P^\sharp)\toplus (0),
\\
(\omega_X\totimes \KKK)/ (\omega_X\totimes\KKK\totimes I)&=& (0)\toplus (-P^\sharp).
\end{eqnarray*}
Our goal is to extend a non-zero section $\bar \xi $ of
$(0)\subset \omega_X\totimes \KKK/ \omega_X\totimes\KKK\totimes I$
to a section $\xi \in \Ho^0(\omega_X\totimes \KKK)$.
By the Formal Function Theorem
\begin{equation*}
\lim_{\longleftarrow}
\Ho^0\left(\frac{\omega_X\totimes \KKK}{\omega_X\totimes \KKK^{(n)}} \right)
\simeq
\lim_{\longleftarrow}\ \frac{
f_*(\omega_X\totimes \KKK)}{ \mm^n_{o,Z}f_*(\omega_X\totimes \KKK)}.
\end{equation*}
Thus, for lifting $\bar \xi$, it is sufficient
to show that the map
\begin{equation*}
\Phi_n: \Ho^0(\omega_X\totimes \KKK/ \omega_X\totimes\KKK^{(n)})
\xrightarrow{\hspace*{20pt}}
\Ho^0(\omega_X\totimes \KKK/ \omega_X\totimes\KKK\totimes I)
\end{equation*}
is surjective for all $n>0$, or equivalently $\Phi_2$ and
\begin{equation*}
\Psi_n:
\Ho^0(\omega_X\totimes \KKK/ \omega_X\totimes\KKK^{(n)})
\xrightarrow{\hspace*{20pt}}
\Ho^0(\omega_X\totimes \KKK/ \omega_X\totimes\KKK^{(n-1)})
\end{equation*}
are surjective for all $n>0$.
We have
\begin{equation*}
0 \to
\omega_X\totimes \left(
\frac{\KKK^{(n-1)}}{\KKK^{(n)}}\right)
\xrightarrow{\hspace*{20pt}}
\frac{\omega_X\totimes \KKK}{ \omega_X\totimes\KKK^{(n)}}
\xrightarrow{\makebox[35pt]{ $\scriptstyle\psi_n$}}
\frac{\omega_X\totimes \KKK}{ \omega_X\totimes\KKK^{(n-1)}}
\to 0.
\end{equation*}
Note that the sheaves
$\omega_X\totimes(\operatorname{im}(\KKK\totimes I \to \KKK))/\KKK^{(2)})$ and
\begin{equation*}
\omega_X\totimes \KKK^{(n-1)} /\omega_X\totimes\KKK^{(n)}
\simeq \tilde S^{n-1}\left(\omega_X\totimes\KKK/\omega_X\totimes\KKK^{(2)}\right)
\end{equation*}
have filtrations with successive subquotients
\begin{equation*}
(-P^\sharp)\totimes \tilde S^{n-1}\left((-P^\sharp)\toplus (0)\right)
\totimes
\begin{cases}
(0)
\\
(-1+2P^\sharp)
\\
(-1+3P^\sharp)
\\
(-1+P^\sharp)
\end{cases}
\end{equation*}
which are all $\ge (-1)$ and hence have vanishing $\Ho^1$.
Thus $\Psi_n=\Ho^0(\psi_n)$ and $\Phi_2$ are onto and so is
$\Phi_n=\Phi_2\circ \Psi_3\circ\cdots \circ \Psi_n$.
\end{scase}

\begin{scase}\label{sde-iP=3+III-sections}
Thus a non-zero section $\bar \xi $ of
$(0)\subset \omega_X\totimes \KKK/ \omega_X\totimes\KKK\totimes I$
induces
a section $\xi \in \Ho^0(\omega_X\totimes \KKK)$ which in turns
induces a generator of $(P^\sharp)$. Let $G:=\{\xi =0\}$.
Then $G\supset 4C$ and $\OOO_H\KKK=\OOO_H(-G)$.
Hence, $\KKK$ is generated by $\xi$ and $\beta$:

\begin{scorollary}\label{scorollary-9-6-8}
The ideal $\KKK$ is a global complete intersection. 
More precisely, $\KKK=(\beta,\xi)$.
\end{scorollary}
Moreover, $\xi$
can be locally written as
$\xi=y_3+(\text{higher degree terms})$. Thus we may assume that there exists
a global section of $\OOO_X$ which is locally written as
$y_1y_3$, i.e. $y_1y_3\in \beta$.
On the other hand by \ref{ge} the general member $D\in |-K_X|$
is given by $y_1+\xi'=0$ for some $\xi'\in (y_2,y_3,y_4)$.
Then replacing $\beta$ with a linear combination of $\beta$ and $(y_1+\xi')\xi$ we may assume that
$y_1y_3$ appears in $\beta$ with arbitrary
coefficient $\lambda$ and $y_4^2$ appears in $\beta$ with coefficient $1$.
In particular, there is a specific section $\beta^\circ$
which does not contain $y_1y_3$ (and contains $y_4^2$).
Then $H$ can be given by the equations $\alpha^\circ=\beta=0$,
where $\alpha^\circ:= \alpha+ y_1^2\beta^\circ$ contains $y_1^2y_4^2$.
\end{scase}

Now applying Computation \ref{computation-lP=3a-III} with $l=3$ or $7$,
we obtain the diagram \ref{main-theorem-divisorial}. The following examples show that
this case does occur.
\end{case}

\begin{example}\label{example-divisorial-lP=3}
Let $Z \subset {\CC}^5_{z_1,\ldots,z_5}$ be
defined by
\begin{eqnarray*}
0&=& z_2^2+z_3+z_4z_5^k-z_1^3,\qquad k\ge 1,\\
0&=& z_1^2z_2^2+z_4^2-z_3z_5.
\end{eqnarray*}
Then $(Z,0)$ is a
threefold singularity of type \type{cD_{5}}.
Let $B \subset Z$ be the $z_5$-axis and
let $f : X \to Z$ be the weighted $(1,1,4,2,0)$-blowup.
The origin
of the $z_3$-chart
is a type \type{(IIA)} point $P$ with $\ell(P)=3$:
\begin{equation*}
\{-y_1^3y_3+y_2^2+y_3^2+y_4(y_1^2y_2^2+y_4^2)^k=0\}/\muu_{4}(1,1,3,2),
\end{equation*}
where $(C,P)$ is the $y_1$-axis.
In the $z_1$-chart we have type \type{(III)} a point.
\end{example}

\begin{example}
\label{example-divisorial-lP=7}
As in \ref{example-divisorial-lP=3}, let $Z \subset \CC^5_{z_1,\ldots,z_5}$ be
defined by
\begin{eqnarray*}
0 &=& z_2^2+z_1^2z_5+ z_3+z_4z_5^k, \qquad k\ge 1,
\\
0 &=& z_3z_5+z_1^5+z_4^2.
\end{eqnarray*}
Then the point $(Z,0)$ is of type \type{cD_{5}}.
Let $B \subset Z$ be the $z_5$-axis and
let $f : X \to Z$ be the weighted $(1,1,4,2,0)$-blowup. 
In the $z_1$-chart $X$ is smooth and
the origin of the $z_3$-chart
is a \type{(IIA)} point $P$ with $\ell(P)=7$:
\begin{equation*}
\{-y_1^7y_3+y_2^2+y_3^2-y_1^2y_4^2+y_4(y_1^5y_3+y_4^2)^k=0\}/\muu_{4}(1,1,3,2),
\end{equation*}
where $(C,P)$ is the $y_1$-axis.
\end{example}

\section{Cases $\ell(P)=4$ and $8$}\label{section-lP=4}
In this section we assume that $\ell(P)\in \{4,\, 8\}$.
We will show that Computation \ref{computation-lP=4+III} is applicable here
and the possibility \ref{main-theorem-conic-bundle} occurs.
\begin{case}
According to \ref{equation-IIA-point} we may write
\begin{equation}\label{equation-lP=4-alpha}
\alpha=y_1^{\ell(P)}y_4+y_2^2+y_3^2+\delta y_4^3+c y_1^2y_4^2+\epsilon y_1y_3y_4+
\zeta y_1^2y_2y_3+\cdots,
\end{equation}
with $\delta,\, c,\, \epsilon,\, \zeta\in \CC\{y_1^4\}$.
It is easy to see that $y_4\in I^{\sharp (2)}$. Hence,
\begin{equation}\label{equation-y14y4}
-y_1^{\ell(P)}y_4\equiv y_2^2+y_3^2+
\zeta y_1^2y_2y_3 \mod I^{\sharp (3)}.
\end{equation}
By Proposition \ref{proposition-cases-lP} in the case $\ell(P)=4$ the variety
$X$ has a type \type{(III)} point $R$ with $i_R(1)=1$ and 
$X$ is smooth outside $P$ in the case $\ell(P)=8$.
\end{case}

\begin{case}
Taking Proposition 
% \ref{proposition-cases-lP}
\ref{proposition-lP=4-XC} into account 
for any $n\ge 1$ we can write
\begin{equation*}
(\gr_C^n \OOO)^\sharp =\bigoplus_{\substack{a+b+2c= n\\ b=0,\ 1}} \OOO_{C^\sharp}\cdot y_2^ay_3^by_4^c,
\end{equation*}
where $a,\, b,\, c\ge 0$, and
\begin{equation}
\label{equation-lP=4-gr1O}
\vcenter{
\xymatrix@R=6pt@C=-3pt{
\gr_C^1\OOO= &(-1+3P^\sharp)\ar@{=}[d]&\toplus& (-1+P^\sharp),\ar@{=}[d]
\\
& \AAA && \BBB
}}
\end{equation}
where $y_2$ (resp. $y_3$) is an $\ell$-basis of
$\AAA$ (resp. $\BBB$) at $P$.
\end{case}
\begin{case}
In the case $\ell(P)=4$
by \cite[Lemma 2.16]{Mori-1988}, since $i_R(1)=1$,
the equation of $X$ at $R$ can be written as follows
\begin{equation}\label{equation-lP=4-gamma}
\gamma(z)=z_1z_4+q_2(z_2, z_3)+q_3(z_1, \dots,z_4),\quad q_3\in (z_2,z_3,z_4)^3,
\end{equation}
where $C$ is the $z_1$-axis and 
$q_2\in \CC\cdot z_2^2+\CC\cdot z_2z_3+\CC\cdot z_3^2$. Hence, $z_4\in I^{(2)}$. 
\end{case}

\begin{case}
Consider the map $\varphi: \tilde S^2\gr_C^1\OOO \hookrightarrow \gr_C^2\OOO$.
Clearly, it is an isomorphism outside $\{P,\, R\}$ (resp. $\{P\}$)
in the case $\ell(P)=4$ (resp. $\ell(P)=8$). 
The equality \eqref{equation-lP=4-gr1O} implies
\begin{eqnarray*}
\tilde S^2\gr_C^1\OOO&=&
(-1+2P^\sharp)\toplus (-1)\toplus (-2+2P^\sharp),
\\
\deg \gr_C^2\OOO &=& -4+\len \coker \varphi\ge -2.
\end{eqnarray*}
Furthermore,
\begin{equation}
\label{equation-lP=4-cokerP-cokerP}
\coker_P \varphi =\CC_{(\ell(P)/4)P}\cdot \overline{(y_1^2y_4)}. 
\end{equation}
Hence, in the case $\ell(P)=4$,
$\coker_R \varphi\neq 0$. Taking Proposition \ref{proposition-lP=4-XC}\ref{proposition-lP=4-XC-1}
into account in this case we obtain
$q_2\neq 0$ (see \eqref{equation-lP=4-gamma}) and
\begin{equation}
\label{equation-lP=4-cokerR-cokerR}
\coker_R \varphi =\CC_R\cdot \bar z_4\simeq \CC.
\end{equation}
Thus in both cases $\ell(P)=4$ and $\ell(P)=8$ we have 
$\deg \gr_C^2\OOO=-2$. 
By Lemma \ref{lemma-lP=3+IIIa}
% Since $\Ho^1(\gr_C^2\OOO)=0$ (see Lemma \ref{lemma-grC}\ref{lemma-grC-3}), 
\begin{equation}
\label{equation-lP=4-gr-2-C-O-1}
\gr_C^2\OOO \simeq \OOO\oplus \OOO(-1)^{\oplus 2}.
\end{equation}
Furthermore, $\gr_C^2\OOO$ has an $\ell$-basis $y_2y_3$, $y_2^2$, $y_4$ at $P^\sharp$.
Thus, 
\begin{equation}
\label{equation-lP=4-gr-2-C-O}
\gr_C^2\OOO=
(0)\toplus(-1+2P^\sharp)\toplus (-1+2P^\sharp),
\end{equation} 
since $\Ho^1 (\gr_C^2\omega)=0$ (cf. Lemma \ref{treating-equation-possibilities-lP=3+III-c}).
\end{case}

\begin{case}
According to \eqref {equation-lP=4-cokerP-cokerP} and \eqref {equation-lP=4-cokerR-cokerR}
\begin{equation}
\label{equation-lP=4-quotient}
\gr_C^2\OOO/\tilde S^2\gr_C^1\OOO\simeq 
\begin{cases}
\CC_P\oplus\CC_R& \text{in the case $\ell(P)=4$,}\\ 
\CC_{2P}& \text{in the case $\ell(P)=8$.} 
\end{cases}
\end{equation}
Let $\FFF$ be the sheaf 
with an $\ell$-structure defined by the  conditions:
\begin{eqnarray*}
&&\tilde S^2\gr_C^1\OOO \subset \FFF\subset \gr_C^2\OOO,
\\
&&\gr_C^2\OOO/\FFF=\CC_P,
\\
&&\gr_C^2\OOO^\sharp/\FFF^\sharp=\OOO^\sharp/(y_1^4)\cdot y_4^2.
\end{eqnarray*}
{}From \eqref{equation-lP=4-gr-2-C-O-1} one can see that 
there are two possibilities: 
\begin{numcases}{\FFF\simeq}
\OOO(-1)^{\oplus 3},\label{lP=4-F-case-1}
\\
\OOO\oplus\OOO(-1)\oplus \OOO(-2).\label{lP=4-F-case-2}
\end{numcases}
\end{case}

\begin{case}{\bf Case \eqref{lP=4-F-case-2}.}
Since $\FFF\subset \gr_C^2\OOO$, by \eqref {equation-lP=4-gr-2-C-O} 
\begin{equation*}
\FFF=(0)\toplus (-1+2P^\sharp) \toplus (-2+2P^\sharp). 
\end{equation*}
Now we treat the cases $\ell(P)=4$ and $\ell(P)=8$ separately.
\end{case}

\begin{slemma}\label{lemma-lP=4-case-does-not-occur}
The case \eqref{lP=4-F-case-2} with $\ell(P)=4$ does not occur.
\end{slemma}

\begin{proof}
Consider the embedding
\begin{equation*}
z_1\cdot(0)\subset \OOO_C(-R)\cdot\FFF \subset \tilde S^2\gr_C^1\OOO
= (-1+2P^\sharp) \toplus(-1)\toplus (-2+2P^\sharp).
\end{equation*}
Clearly, the image in the third summand is zero and the projection to the second summand is 
multiplication by a constant. Moreover, if this constant is zero, then the image of $z_1\cdot(0)$
is contained in $(-1+2P^\sharp)$. 
In other words, the summand $(0)\subset \FFF\subset \gr_C^2\OOO$ 
is contained in $(2P^\sharp)$ which is impossible by \eqref{equation-lP=4-gr-2-C-O}.

By changing $\ell$-splitting as follows 
\begin{equation*}
z_3 \longmapsto z_3+(\const ) z_2,\quad y_3 \longmapsto y_3+(\const) y_1^2y_2, 
\end{equation*}
one can assume that $q_2\in \CC^*\cdot z_2z_3$ and so $(0)\ni \overline{z_1z_4}=\overline{z_2z_3}$.
Furthermore, 
$\FFF\supset (0)=\OOO_C\cdot z_4$ at $R$ by changing coordinates as $z_4 \mapsto z_4+\cdots$.
Since $\FFF\subset \gr_C^2\OOO$, $(0)$ is sent isomorphically to 
$(0)\subset \gr_C^2\OOO$. 
We have the inclusion $\gr_C^2\OOO\supset \OOO_C\cdot \bar\beta=\AAA\totimes \BBB(R)$ 
(see \eqref{equation-lP=4-gr1O}). Hence, $\bar\beta=\nu y_2y_3$ at $P^\sharp$, where $\nu$ is a unit.
\begin{sclaim}
$\bar\beta \gr_C^1\OOO$ is an $\ell$-subbundle
of $\gr_C^3\OOO$ and and the natural map
$\AAA^{\totimes 3}\to \gr_C^3\OOO/\bar\beta \gr_C^1\OOO$
induces the following $\ell$-exact sequence
\begin{equation}
\label{iP=4-equation-exact-sequence-l}
\vcenter{
\xymatrix@R=6pt@C=20pt{
0\ar[r]&\AAA^{\totimes 3}(4P^\sharp) \ar[r]\ar@{=}[d]&\gr_C^3\OOO/\bar\beta \gr_C^1\OOO 
\ar[r]& \BBB^{\totimes 3}(4P^\sharp)\ar@{=}[d] \ar[r]& 0,
\\
&(P^\sharp)&&(-2+3P^\sharp)
} }
\end{equation}
where $y_2y_4$ \textup(resp. $y_3y_4$\textup) is an $\ell$-basis of $\AAA^{\totimes 3}(4P^\sharp)$
\textup(resp. $\BBB^{\totimes 3}(4P^\sharp)$\textup).
\end{sclaim}

\begin{proof}
To check the assertion at $R$ we apply Proposition \ref{proposition-lP=4-XC}\ref{proposition-lP=4-XC-5}
with $m=1$ and $\bar\beta=z_4$, and note that $\gr_C^3\OOO/\bar\beta \gr_C^1\OOO=\OOO_Cz_2^3\oplus \OOO_C z_3^3$. 
At $P^\sharp$, we note that $\bar\beta=\nu y_2y_3$ and 
use Proposition \ref{proposition-lP=4-XC}\ref{proposition-lP=4-XC-4} with $h=\alpha$
to show that $\gr_C^3\OOO$ has $\ell$-basis $y_2^3$, $y_2^2y_3$, $y_2y_4$, $y_3y_4$.
By \eqref{equation-y14y4}
\begin{equation*}
y_1^{4}y_4+y_2^2+y_3^2+\zeta y_1^2\bar\beta=0.
\end{equation*}
Then $\gr_C^3\OOO/\bar\beta\gr_C^1\OOO$ has an $\ell$-free $\ell$-basis 
$y_2y_4$, $y_3y_4$ because 
$y_3^2y_2\equiv -y_2^3-y_1^4y_2y_4$, and we have $y_2^3\equiv -y_1^4y_2y_4\mod (\bar\beta)$ 
and $y_2y_4\equiv -y_2^3/y_1^4 \mod (\bar\beta)$. This shows the exactness because 
$y_3^3\equiv -y_1^4y_3y_4\mod (\beta)$.
\end{proof}
To complete the proof of Lemma \ref{lemma-lP=4-case-does-not-occur} we note that
the sequence \eqref{iP=4-equation-exact-sequence-l}
implies that $\Ho^1(\gr_C^3\OOO/\bar\beta \gr_C^1\OOO )\neq 0$.
This contradicts Lemma \ref{lemma-grC}.
Thus the case \eqref{lP=4-F-case-2} with $\ell(P)=4$ does not occur.
\end{proof}

\begin{slemma}\label{lemma-lP=8-case-does-not-occur}
The case \eqref{lP=4-F-case-2} with $\ell(P)=8$ does not occur.
\end{slemma}

\begin{proof}
We have $0\neq \bar\beta \in \Ho^0((0))\subset \Ho^0(\FFF)$.
Since $\bar\beta \notin \Ho^0(\tilde S^2\gr_C^1\OOO)$ and 
$\FFF/\tilde S^2\gr_C^1\OOO=\CC\cdot \overline{y_1^6y_4}$, we have 
\begin{equation}
\label{equation-lP=4-barbeta}
\bar\beta=(\cdots )y_2^2+(\cdots )y_2y_3+(\unit)y_1^6y_4. 
\end{equation}
{}From the following relation
\begin{equation*}
\bar\beta\cdot (-1) \subset \FFF(-4P^\sharp)\subset \tilde S^2\gr_C^1\OOO
= (-1+2P^\sharp) \toplus(-1)\toplus (-2+2P^\sharp)
\end{equation*}
we see that the image of $y_1^4\bar\beta$ in the 
third summand is zero and the projection to the second summand is multiplication by a 
constant. Moreover, if this constant is zero, then the image of $y_1^4\cdot(0)$
is contained in $(-1+2P^\sharp)$. 
In other words, the summand $(0)\subset \FFF\subset \gr_C^1\OOO$ 
is contained in $(2P^\sharp)$ which is impossible by \eqref{equation-lP=4-gr-2-C-O}.
Therefore, 
\begin{equation*}
y_1^4\bar\beta =(\cdots)y_2^2+(\unit) y_2y_3.
\end{equation*}
Then \eqref{equation-lP=4-barbeta} implies
\begin{equation*}
y_1^{10}y_4\equiv (\cdots)y_2^2+(\unit) y_2y_3 \mod I^{(3)}.
\end{equation*}
On the other hand, $y_4$, $y_2^2$, $y_2y_3$ form an $\ell$-basis of 
$\gr_C^2\OOO$, a contradiction.
This proves Lemma \ref{lemma-lP=8-case-does-not-occur}.
\end{proof}

\begin{case}{\bf Case \eqref{lP=4-F-case-1}.}\label{case-conic-bundle}
If the coefficient of $y_1^2y_4$ in $\bar\beta$ is zero,
then $\bar\beta \in \Ho^0(\FFF)$. But
in our case $\Ho^0(\FFF)=0$ which gives us a contradiction. 

Thus for a general choice of $\beta\in \Ho^0(\OOO_X)$
at $P$ we can write $\bar\beta=\nu y_2y_3+\eta y_1^2y_4+\cdots$ and so 
\begin{equation*}
\beta=\theta y_4^2 +\nu y_2y_3+\eta y_1^2y_4+\cdots,
\end{equation*}
where $\theta,\, \nu,\, \eta$ are units. 
This means that $y_1^2y_4\in \beta$.
Since $\operatorname{h}^0(\gr_C^2\OOO)=1$, the ratio of 
the coefficients $\nu$ and $\eta$ is fixed.
On the other hand, the ratio of 
the coefficients of $\nu$ and $\theta$
is general \cite[Lemma 3.1.1]{Mori-Prokhorov-IIA-1}.
Hence the ratio of coefficients $\theta$ and $\eta$ 
can be chosen general. 
Then we apply Computation \ref{computation-lP=4+III}.
One can see that the graph \eqref{graph-diagram-non-normal-lP=4+III} corresponds
to a conic bundle. 
We obtain the diagram \ref{main-theorem-conic-bundle}.
Examples \ref{example-conic-bundle-lP=4+III}
and \ref{example-conic-bundle-lP=8} below show that both possibilities $\ell(P)=4$ and $8$ do occur.
\end{case}

\begin{example}\label{example-conic-bundle-lP=4+III}
Let $X$ be the the hypersurface of weighted degree $10$ in 
the weighted projective space $\PP(1,1,3,2,4)_{x_1,x_2, x_3, x_4, w}$ 
% with coordinates $x_1,x_2, x_3, x_4, w$
given by the equation 
% $w\phi_6 -x_1^6\phi_4=0$, where 
\begin{equation*}w\phi_6 -x_1^6\phi_4=0,\quad \text{}\quad
\begin{array}{lll}
\phi_6&:=&x_1^4x_4+x_3^2+x_2^2w+\delta x_4^3,
\\
\phi_4&:=&x_4^2+\nu x_2x_3+\eta x_1^2x_4+\mu x_1^3x_2
\end{array}
\end{equation*}
% $X\subset \PP(1,1,3,2,4)$ be a small analytic neighborhood of $C=
% \{\text{$(x_1,w)$-line}\}$ given by the equation
% $x_1^6\phi_4-w \phi_6 =0$, where 
% Let $\PP$ be the weighted projective space $\PP(1,1,3,2,4)$ with coordinates $x_1,x_2, x_3, x_4, w$.
% Let $C$ be the $(x_1,w)$-line and let 
% \begin{equation*}
% X=\{w\phi_6 -x_1^6\phi_4=0\}\subset \PP 
% \end{equation*}
% be the hypersurface of weighted degree $10$, where 
% \begin{eqnarray*}
% \phi_6&:=&x_1^4x_4+x_3^2+x_2^2w+\delta x_4^3,
% \\
% \phi_4&:=&x_4^2+\nu x_2x_3+\eta x_1^2x_4+\mu x_1^3x_2
% \end{eqnarray*}
(for simplicity we assume that the coefficients $\delta$, $\nu$, $\eta$ are 
general).
Regard $X$ as a small analytic neighborhood of $C$.
In the affine chart $U_w:=\{w\neq 0\}\simeq \CC^4/\muu_{4}(1,1,3,2)$
the variety $X$ is given by 
\begin{equation*}
\phi_6(y_1,y_2,y_3,y_4, 1) - y_1^6\phi_4(y_1,y_2,y_3,y_4, 1)=0
\end{equation*}
and $C$ is the $y_1$-axis.
Clearly, it has the form \eqref{equation-lP=4-alpha}.
So, the origin $P\in (X,C)$ is a type \type{(IIA)} point with $\ell(P)=4$. 

In the affine chart $U_1:=\{x_1\neq 0\}\simeq \CC^4$
the variety $X$ is defined by 
\begin{equation*}
w\phi_6(1,z_2,z_3,z_4, w) - \phi_4(1,z_2,z_3,z_4, w)=0.
\end{equation*}
If $\mu\neq 0$, then $X$ is smooth outside $P$, i.e. $(X,C)$ is as 
in the case \cite[(1.1.4)]{Mori-Prokhorov-IIA-1}. If $\mu=0$, then 
$(X,C)$ has a type \type{(III)} point at $(0,0,0,\eta)$.
% with $\ell(R)=1$.
% at $\{x_1=1,\, w=\eta\}\in C$.
% (see \cite[Lemma 2.16]{Mori-1988}). 

Consider the surface $H=\{\phi_6=\phi_4=0\}\subset X$.
Let $\psi : H^{\n}\to H$ be the normalization (we put $H^{\n}=H$ if $H$ is normal)
and let $C^{\n}:=\psi^{-1}(C)$. Near $P$ the surface 
$H$ has the form \cite[9.3]{Mori-Prokhorov-IIA-1} (resp. 
\ref{computation-lP=4+III}) if 
$\mu \neq 0$ (resp. $\mu =0$). In particular, the singularities of $H^{\n}$ are rational.
Note that $H$ is a fiber of the fibration $\pi: X\to D$ over a small disk around the origin 
given by the rational function $ \phi_4/w =\phi_6/x_1^6$
which is regular in a neighborhood of $C$.
By the adjunction formula $\OOO_{X}(K_X)=\OOO_{X}(-1)$. 
Hence, 
\begin{equation*}
-K_H\cdot C=-K_X\cdot C=\OOO_{\PP}(1)\cdot C=\textstyle \frac14.
\end{equation*}

\begin{sclaim}
\begin{enumerate}[leftmargin=20pt]
\item 
If $\mu \neq 0$, then $H$ is smooth outside $P$.

\item
Assume that $\mu = 0$.
Let $P_1\in C$ be the point $\{4\eta^2w=\nu^2 x_1^4\}$.
Then $H$ is singular along $C$, the curve
$C^{\n}$ is irreducible and rational, and $\psi_C:= C^{\n}\to C$ is a double cover
branched over $\{P,\, P_1\}$.
Moreover, $\psi^{-1}(P)$ is the only singular point of $H^{\n}$.
\end{enumerate}
\end{sclaim}
\begin{proof}
Direct computations show that 
$P_1\in H$ is a pinch point \textup(see \xref{definition-pinch-point}\textup) and any $Q\in C\setminus \{P,\, P_1\}$ 
is a double normal crossing point of $H$.
\end{proof}

% \begin{sclaim}
% 
% \end{sclaim}
% \begin{proof}
% The surface $H^{\n}$ has only rational singularities. Thus it is sufficient to
% check only that the numerical pull-back of $4C^{\n}$ (resp. $8C^{\n}$) to the minimal resolution has
% integral coefficients.
% \end{proof}

\begin{sclaim}
\label{claim-conic-bundle-C-Q-Cartier}
If $\mu=0$ \textup(resp. $\mu\neq 0$\textup), then 
$4C^{\n}$ \textup(resp. $8C^{\n}$\textup) is a Cartier divisor on $H^{\n}$. Moreover, 
$(C^{\n})^2=0$.
\end{sclaim}
\begin{proof}
We consider only the case where $H$ is not normal, i.e. $\mu=0$.
The case $\mu\neq 0$ is easier and left to the reader.
Let $V\subset \PP(1,1,3,2,4)$ be the weighted hypersurface given by $x_4=0$
and let $M:= H\cap V$.
We have $M=\{x_3^2+w x_2^2=x_2x_3=x_4=0\}$.
Let $\Gamma$ be the line $\{x_3=x_4=w=0\}$
and let $\Gamma^{\n}$ be its preimage on $H^{\n}$.
Then 
$\psi^* 2M= 4 C^{\n}+ 2\Gamma^{\n}$.
Since $2M$ is Cartier near $C$ and $\Gamma^{\n}$ is contained in the smooth locus of $H^{\n}$,
the divisor $4 C^{\n}$ is Cartier on $H^{\n}$. Further, 
by the projection formula
\begin{equation*}
\psi^* 2M\cdot C^{\n} =4 V\cdot C =2.
\end{equation*}
Since $\Gamma$ is smooth and $C^{\n}\to C$ is \'etale over the point $\Gamma\cap C$,
the curves $\Gamma^{\n}$ and $C^{\n}$ meet each other transversely 
at one point which is a smooth point of $H^{\n}$.
Hence, $\Gamma^{\n}\cdot C^{\n}=1$ and so 
\begin{equation*}
\label{equation-conic-bundle-H-c2}
4 (C^{\n})^2= \psi^* 2M \cdot C^{\n} - 2\Gamma^{\n}\cdot C^{\n}= 2-2=0.\qedhere
\end{equation*}
\end{proof}

\begin{sclaim}\label{claim-conic-bundle-surface-H}
There exists a rational curve fibration $f_H: H\to B$, where $B\subset \CC$
is a small disk around the origin, such that $C=f_H^{-1}(0)_{\red}$.
\end{sclaim}

\begin{proof}
Using the explicit description of the minimal resolution
(see \cite[9.3]{Mori-Prokhorov-IIA-1}, \eqref{graph-diagram-non-normal-lP=4+III}) and Claim \ref{claim-conic-bundle-C-Q-Cartier},
one can see that the contraction exists on $H^{\n}$.
Then, clearly, it descends to $H$.
\end{proof}

\begin{sclaim}
One has $\Ho^1(\hat X,\OOO_{\hat X})=0$,
where $\hat X$ denotes the completion of $X$ along $C$.
\end{sclaim}
\begin{proof}
Consider the case $\mu=0$ (the case $\mu\neq 0$ is similar and easier).
By Claim \ref{claim-conic-bundle-surface-H} \ $4C^{\n}=\operatorname{div} (\varphi)$ 
for some regular function $\varphi\in \Ho^0(\OOO_{H^{\n}})$. Since $\varphi|_{C^{\n}}=0$, 
this function descends to $H$ and defines a Cartier divisor $\CCC$ on $H$ such that
$\psi^* \CCC=4C^{\n}$. Consider the standard injection
$\theta: \OOO_H\to \psi_* \OOO_{H^{\n}}$. Then there is the following commutative diagram
\begin{equation*}
\label{equation-conic-bundle-diagram}
\begin{xy}
\xymatrix@C=39pt{
& I_C\ar@{^{(}->}[d] & I_{C^{\n}}\ar@{^{(}->}[d]
\\
0\ar[r]&\OOO_H\ar@{->>}[d] \ar[r]^{\theta}& \psi_* \OOO_{H^{\n}}\ar@{->>}[d]\ar[r] &\coker(\theta)\ar[d]^{\simeq}\ar[r]&0 
\\
0\ar[r]&\OOO_C \ar[r]^{\theta}& \psi_* \OOO_{C^{\n}}\ar[r] &\psi_* \OOO_{C^{\n}}^{\langle\iota=-1\rangle}\ar[r]&0 
} 
\end{xy}
\end{equation*}
where $\OOO_{C^{\n}}^{\langle\iota=-1\rangle}$ is the anti-invariant part with respect to the Galois 
involution $\iota: C^{\n}\to C^{\n}$.
Since the last row in this diagram splits
and $\Ho^1(\OOO_{C^{\n}})=0$, we have $\Ho^1(\coker(\theta))=0$.
Using the snake lemma we see that the multiplication by $\varphi$ induces the following diagram 
\begin{equation*}
\begin{xy}
\xymatrix@C=39pt@C=33pt{
&0\ar[r]&\OOO_H\ar@{^{(}->}[d]^{\cdot \varphi} \ar[r]^{\theta}& \psi_* \OOO_{H^{\n}}\ar@{^{(}->}[d]^{\cdot \varphi}
\ar[r] &\coker(\theta)\ar[d]^{\cdot \varphi=0}\ar[r]&0 
\\
& 0\ar[r]&\OOO_H\ar@{->>}[d] \ar[r]^{\theta}& \psi_* \OOO_{H^{\n}}\ar@{->>}[d]\ar[r] &\coker(\theta)\ar[d]^{\simeq}\ar[r]&0 
\\
0\ar[r]& \coker(\theta)\ar[r]&\OOO_{\CCC}\ar[r]& \psi_* \OOO_{4C^{\n}}\ar[r] &\coker(\theta)\ar[r]&0
}
\end{xy}
\end{equation*}
Since $\Ho^1(\coker(\theta))=0$, from the last row we see 
$\Ho^1(\OOO_{\CCC})\simeq \Ho^1(\OOO_{4C^{\n}})$.
On the other hand, $4C^{\n}$ is a fiber of a rational curve fibration.
Hence, $\Ho^1(\OOO_{\CCC})\simeq \Ho^1(\OOO_{4C^{\n}})=0$.
Similar arguments show that $\Ho^1(\OOO_{m\CCC})=0$ for any $m>0$.
Then by the Formal Function Theorem $\Ho^1(\hat H, \OOO_{\hat H})=0$,
where $\hat H$ is the completion of $H$ along $C$.
Applying the Formal Function Theorem again we obtain $\Ho^1(\hat X, \OOO_{\hat X})=0$.
\end{proof}
\begin{sclaim}
The contraction $f_H: H\to B$ extends to a contraction $\hat f: \hat X\to \hat Z$.
\end{sclaim}
\begin{proof}
Since $\Ho^1(\OOO_{\hat X})=0$, from the exact sequence
\begin{equation*}
0 \xrightarrow{\hspace*{20pt}} \OOO_X \xrightarrow{\hspace*{20pt}} \OOO_X (H) \xrightarrow{\hspace*{20pt}} \OOO_H (H)\xrightarrow{\hspace*{20pt}} 0 
\end{equation*}
we see that the map $\Ho^0(\OOO_{\hat X} (\hat H))\to \Ho^0(\OOO_{\hat H} (\hat H))$
is surjective.
Hence there exists a divisor $\hat H_1\in |\OOO_{\hat X}|_{\hat C}$ such that $\hat H_1|_{\hat H}=\hat \CCC$. Then the divisors 
$\hat H$ and $\hat H_1$ define a contraction $\hat f: \hat X\to \hat Z$.
\end{proof}

\begin{sclaim}\label{claim-contraction-exists}
There exists a contraction $f:X\to Z$ that approximates $\hat f: \hat X\to \hat Z$.
\end{sclaim}
\begin{proof}
Let $F$ be the scheme fiber of $f_H: H\to B$ over the origin.
The above arguments shows that the deformations of $F$ are  unobstructed.
Therefore the corresponding component of the Douady space is smooth and two-dimensional.
This allow us to produce a contraction $f: X\to Z$.
\end{proof}
\end{example}

\begin{example}
\label{example-conic-bundle-lP=8}
Similar to Example \ref{example-conic-bundle-lP=4+III}, let
$X\subset \PP(1,1,3,2,4)$ be a small analytic neighborhood of $C=
\{\text{$(x_1,w)$-line}\}$ given by the equation
$x_1^6\phi_4-w \phi_6 =0$, where 
\begin{eqnarray*}
\phi_6&:=&x_3^2+x_2^2w+\delta x_4^3+cx_1^2x_4^2,
\\
\phi_4&:=&x_4^2+\nu x_2x_3+\eta x_1^2x_4.
\end{eqnarray*}
It is easy to check that $P:=(0:0:0:0:1)$ is the only 
singular point of $X$ on $C$ and it is a type \type{(IIA)} point with $\ell(P)=8$.
The rational function $\phi_4/w=\phi_6/x_1^6$ near $C$ defines a fibration
whose central fiber $H$ is given by $\phi_4=\phi_6 =0$.
Existence of a contraction 
$f: X\to Z$ can be shown similar to Claim \ref{claim-contraction-exists}.
Near $P$ the surface $H$ has the following form which can be reduced to \ref{computation-lP=4+III}:
\begin{equation*}
-c\eta y_1^4y_4+ y_3^2+y_2^2+\delta y_4^3-c\nu y_1^2 y_2y_3=\phi_4=0.
\end{equation*}
\end{example}

\begin{subexample-remark} 
\label{example-conic-bundle-normal-H}
In a similar way we can construct an example of a $\QQ$-conic bundle with $\ell(P)=5$
and normal $H$
\cite[(1.1.4)]{Mori-Prokhorov-IIA-1}.
Consider $X\subset \PP(1,1,3,2,4)$ given by $w\phi_6-x_1^6\phi_4=0$, where
\begin{eqnarray*}
\phi_6&:=&x_1^5 x_2+x_2^2w+x_3^2+\delta x_4^3+cx_1 ^2 x_4^2
\end{eqnarray*}
and $\phi_4$ is as in \ref{example-conic-bundle-lP=4+III}.
In the affine chart $U_w\simeq \CC^4/\muu_{4}(1,1,3,2)$
the origin $P\in (X,C)$ is a type \type{(IIA)} point with and $\ell(P)=5$. 
It is easy to see that $X$ is smooth outside $P$.
The rational function $\phi_4/w=\phi_6/x_1^6$ defines 
a fibration on $X$ near $C$ with central fiber $H=\{\phi_4=\phi_6=0\}$.
\end{subexample-remark}

\section{Appendix}
In this section we collect computations of resolutions of 
(non-normal) surface singularities appearing as general members 
$H\in |\OOO_X|$. The techniques is very similar to that used in 
\cite[\S 9]{Mori-Prokhorov-IIA-1}
\begin{assumption}
\label{notation-blowup-1}
Let $W:= \CC^4_{y_1,\dots,y_4}/\muu_4(1,1,3,2)$ and let $\sigma$ be the weight $\frac14(1,1,3,2)$.
Let $P\in X$ be a three-dimensional terminal singularity of type \type{cAx/4}
given in $W$ there by the equation $\alpha=0$ with
\begin{equation}\label{equation-alpha-computations}
\alpha=y_1^ly_j+y_2^2+y_3^2+\delta y_4^{2k+1}+c y_1^2y_4^2+\epsilon y_1y_3y_4
+y_2\alpha'+\alpha'',
\end{equation}
where $j=3$ or $4$,\ $l\in \ZZ_{>0}$,\ 
$c, \epsilon\in \CC$, \ $\delta\in \CC^*$, \ 
$\alpha'\in (y_2,\, y_3,\, y_4)$,\
$\alpha''\in (y_2,\, y_3,\, y_4)^2$,\
$\sigmaord (\alpha')= 5/4$,\ 
$\sigmaord (\alpha'')> 3/2$,\ 
$k\ge 1$, and $2k+1$ is the smallest exponent of $y_4$
appearing in $\alpha$.
We usually assume that all the summands in \eqref{equation-alpha-computations}
have no common terms.
\end{assumption}

\begin{sconstruction}\label{construction-w-blowup-X}
Consider the weighted $\sigma$-blowup $\Phi: \tilde W\to W$. Let
$\tilde X$ be the proper transform of $X$ on $\tilde W$ and
$\Pi\subset \tilde W$ be the $\Phi$-exceptional divisor. Then $\Pi\simeq \PP(1,1,3,2)$ and
$\OOO_{\Pi}(\Pi)\simeq \OOO_{\PP}(-4)$.
Put
\begin{equation}\label{equation-Lambda-computation-2}
\begin{aligned}
O&:=(1:0:0:0),\quad Q:=(0:0:1:0)\in \Pi,
\\
\Lambda&:=\{y_2=\alpha_{\sigma=6/4}=0\}\subset \Pi.
\end{aligned}
\end{equation}
Let $\tilde X\subset \tilde W$
be the proper transform of $X$.
\end{sconstruction}

\begin{sclaim}\label{claim-4-construction-blowup}
$\Sing(\tilde X)$ consists of the curve $\Lambda$, the point $Q$, and the point
$Q_1:=(0:0:0:1)$ \textup($Q_1\notin \Lambda$ only if $k=1$\textup).
\end{sclaim}

\begin{sclaim}\label{claim-singularities-Lambda}
$\tilde X$ has singularity of
type \type{cA_1} at a general point of $\Lambda$.
\end{sclaim}

\begin{proof}
Let $D\in |-K_X|$ be a general member
and let $F$ be a general hyperplane section of $X$ passing through $0$. 
We may assume that $D$ is given by
$y_1+y_2+\cdots $ (see \ref{ge}) and 
$F$ is given by $y_1y_3+\cdots=0$.
It is easy to compute 
\begin{equation*}
\Phi^*\left( K_X+D+\textstyle\frac 12 F\right)= 
K_{\tilde X}+\tilde D+\textstyle\frac 12 \tilde F+E\qq 0,
\end{equation*}
where $E=\left(\Pi|_{\tilde X}\right)_{\red}=\{y_2=0\}\subset \Pi$, so $E\simeq \PP(1,3,2)$
with natural coordinates $y_1$, $y_3$, $y_4$. 
By the adjunction formula \cite[Th. 16.5]{Utah}
\begin{equation}\label{equation-adjunction}
\textstyle
\left.
\left(K_{\tilde X}+\tilde D+\frac 12 \tilde F+E\right)\right|_E=K_{E}+\tilde D|_E+\frac 12 \tilde F|_E+\Diff_E(0)\qq 0,
\end{equation}
where $\Diff_E(0)$ is an effective divisor supported on $\Lambda$.
Let $G:=\{y_1=0\}\subset E$. Then $G$ is a generator of $\Cl(E)\simeq \ZZ$.
It is easy to see that $\tilde D|_E\sim G$, $\tilde F|_E\sim 4G$, and
$\Lambda\sim 6G$. 
% According to 
By \eqref{equation-adjunction}
we have
$\Diff_E(0) \qq 3 G$, i.e. $\Diff_E(0)=\frac 12 \Lambda$.
By the inversion of adjunction $K_{\tilde X}+E$ is plt 
at a general point of $\Lambda$ \cite[Th. 17.6]{Utah}.
Then by \cite[Th. 16.6]{Utah} the variety $\tilde X$ has singularity of
type \type{cA_1} at a general point of $\Lambda$.
\end{proof}

\begin{assumption}\label{notation-blowup}
In the notation of \xref{notation-blowup-1}
consider a non-normal
surface singularity $H\ni 0$ given in $W$ by
two $\sigma$-semi-invariant equations $\alpha=\beta=0$.
We assume that the following conditions are satisfied
\begin{itemize}[leftmargin=20pt]
\item
$H$ is singular along $C:=\{\text{$y_1$-axis}\}/\muu_4$ and smooth outside $C$,
\item
$\alpha$ satisfies the assumptions of \xref{notation-blowup-1},
\item
$\wt \beta\equiv 0\mod 4$,
\item
$y_4^2$ appears in $\beta$ with coefficient $1$,
\item
$y_2y_3$ appears in $\beta$ with coefficient $\nu$ which can be taken general,
\item
the normalization of $H$ has only rational singularities
and, for any resolution, the total transform of $C$ has only normal crossings. 
\end{itemize}

\begin{scase}\label{computations-notation}
We can write the equations of $H$ in the following form
\begin{eqnarray*}
\alpha&=&y_1^ly_j+y_2^2+y_3^2+\delta y_4^{2k+1}+c y_1^2y_4^2+\epsilon y_1y_3y_4
+y_2\alpha'+\alpha'',
\\
\beta&= &y_4^2+\nu y_2y_3+\lambda y_1y_3+\eta y_1^2y_4+y_2\beta'+\beta'',
\end{eqnarray*}
where $\alpha$ is as in  \ref{notation-blowup-1},\quad 
$\eta, \nu, \lambda\in \CC$, \ 
$\beta',\, \beta''\in (y_2,\, y_3,\, y_4)$,\
$\sigmaord (\beta')= 3/4$,\ and 
$\sigmaord (\beta'')> 1$.
We usually assume that all the summands in $\beta$
have no common terms. Then
$\beta'\in (y_1y_4,\, y_1y_2,\, y_2^2, y_2y_4)$.
\end{scase}
\end{assumption}

\begin{sremark}\label{remark-computation-normality}
Since $H$ is singular along $C$, we have $y_1^sy_r\notin \beta$ for any $r\neq j$ and any $s$.
Hence $\lambda\eta=0$. Moreover, if $\lambda\neq 0$, then $j=3$
and if $\eta\neq 0$, then $j=4$.
We also may assume that $\beta''\in (y_2,\, y_3,\, y_4)^2$.
\end{sremark}

\begin{sconstruction}\label{notation-computations--sing}
As in \xref{construction-w-blowup-X} consider the weighted 
$\sigma$-blowup $\Phi: \tilde W\to W$. 
Let $\tilde H\subset \tilde W$
(resp. $\tilde C\subset \tilde W$) be the proper transform of $H$
(resp. $C$).
Clearly, $\tilde C\cap \Pi=\{ O\}$.
Denote (scheme-theoretically)
\begin{equation*}
\Xi:=\tilde H\cap \Pi = \{y_2^2=\beta_{\sigma=1}=0\} \subset \Pi.
\end{equation*}
The surface $\tilde H$ is smooth outside $\tilde C\cup \Supp(\Xi)$ and the set $\tilde C\cup \Supp(\Xi)$ 
is covered by two affine charts in $\tilde W$
\begin{equation*}
U_1=\{y_1\neq 0\}\simeq \CC^4,\qquad U_3=\{y_3\neq 0\}\simeq \CC^4/\muu_3(1,1,2,2).
\end{equation*}
Let $\varphi: \hat H\overset{\tau}{\longrightarrow} \tilde H^{\n}\overset{\tilde \psi}{\longrightarrow} \tilde H$ be
the composition of the normalization and the minimal resolution
and let $\hat \Xi_i\subset \hat H$ be
the proper transform of $\Xi_i$.
Let $\tilde C^{\n}=\tilde \psi^{-1}(\tilde C)_{\red}$ and
let $\hat C\subset \hat H$ be the proper transform of $\tilde C^{\n}$.
\end{sconstruction}

\begin{sclaim}[{\cite[9.1.4]{Mori-Prokhorov-IIA-1}}]\label{claim-1-construction-blowup}
Any irreducible component $\Xi_i$ of $\Xi$ is a smooth rational curve
passing through $Q$.
Moreover, $\Xi=2\Xi_1$ \textup(resp. $\Xi=2\Xi_1+2\Xi_2$,\ $\Xi=4\Xi_1$\textup)
if and only if $\lambda\neq 0$ \textup(resp. $\lambda=0$ and $\eta\neq 0$,
$(\lambda, \eta)= (0,0)$\textup).
\end{sclaim}

\begin{sclaim}[{\cite[9.1.5]{Mori-Prokhorov-IIA-1}}]
\label{claim-2-construction-blowup}
The point $Q\in \tilde H$ is Du Val of type \type{A_2}.
In particular, $\tilde H$ is normal outside $\tilde C$.
\end{sclaim}

\begin{sclaim}\label{claim-3-construction-blowup-a}
If at least one of the constants $\lambda$ or $\eta$ is non-zero, then 
the singular locus of $\tilde H$ coincides with 
$\bigl(\Supp(\Xi)\cap \Lambda\bigr)\cup \{Q\}\cup \tilde C$.
\end{sclaim}
\begin{proof}
Direct computations. 
\end{proof}

\begin{sremark}
Let $\psi: H^{\n}\to H$ be the normalization
and let $C^{\n}:=\psi^{-1}(C)_{\red}$.
Since
$H$ has double singularities at a general point of $C$,
the map $\psi_C: C^{\n}\to C$ is either birational or a double cover.
In particular, $C^{\n}$ has at most two components.
\end{sremark}

\begin{sdefinition}\label{definition-pinch-point}
A surface singularity $0\in S$ is called a \emph{pinch point} if it is analytically isomorphic to
\begin{equation*}
0\in \{z_2^2+z_1z_3^2=0\}\subset \CC^3. 
\end{equation*}
\end{sdefinition}

\begin{sremark}
The singular locus of a surface $S$ near a pinch point $0$ is a smooth curve $C$,
the normalization $\psi: S^{\n}\to S$ of $S$ is smooth, and $\psi_C: \psi^{-1}(C)\to C$
is a double cover ramified over $0$.
\end{sremark}

\begin{sclaim}\label{claim-construction-blowup-Du-Val}
The singularities of
$\tilde H^{\n}$ are Du Val outside the preimage of $\tilde C$.
If moreover $\beta$ contains either $y_1y_3$ or $y_1^2y_4$,
then the singularities of $\tilde H^{\n}$ are Du Val everywhere.
\end{sclaim}
\begin{proof}
By Claim \ref{claim-2-construction-blowup}
\ $\tilde H$ has a Du Val singularity at $Q$.
Note that near $O$ the surface $\tilde H$ is a hypersurface
singularity of the form $x_2^2=\phi(x_1,x_3)$, where $\tilde C$ is the $x_1$-axis.
The normalization $\tilde \psi: \tilde H^{\n}\to \tilde H$ can be obtained as a sequence of
successive blowups over $\tilde C$.
In particular, $\tilde H^{\n}$ has only hypersurface singularities.
Finally we note that a two-dimensional
rational Gorenstein singularity must be Du Val.
\end{proof}

\begin{sclaim}{\cite[9.1.9]{Mori-Prokhorov-IIA-1}}
\label{claim-5-equation-notation-blowup}
$K_{\tilde H}=\Phi^*K_H-\frac 34 \Xi$.
\end{sclaim}

\begin{sclaim}\label{claim-computation-Xi-n}
Assume that the singularities of
$\tilde H^{\n}$ are Du Val
\textup(cf. Claim \xref{claim-construction-blowup-Du-Val}\textup).
Write $K_{\tilde H^{\n}}=\tilde \psi^* K_{\tilde H}-\Upsilon$, where
$\Upsilon$ is the effective divisor defined by the conductor ideal.
\begin{itemize}
\item
If $\Xi=2\Xi_1$, then $\hat \Xi_1^2=-4+\tau^*\Upsilon \cdot \hat \Xi_1$.
\item
If $\Xi=2\Xi_1+2\Xi_2$, then $\hat \Xi_i^2=-3+\tau^*\Upsilon \cdot\hat \Xi_i$.
\end{itemize}
\end{sclaim}

\begin{proof}
Consider, for example, the first case $\Xi=2\Xi_1$.
As in \cite[Claim 9.1.10]{Mori-Prokhorov-IIA-1},
$K_{\tilde H}\cdot \Xi_1=2$.
Since $\tilde H$ has only Du Val singularities, we have 
\begin{equation*}
K_{\hat H}=\varphi^*K_{\tilde H} - \tau^*\Upsilon,\qquad
K_{\hat H}\cdot \hat \Xi_1= K_{\tilde H}\cdot \Xi_1
-\hat \Xi_1\cdot\tau^*\Upsilon.
\end{equation*}
Therefore, $\hat \Xi_1^2=-2- K_{\hat H}\cdot \hat \Xi_1=
-4+\hat \Xi_1\cdot\tau^*\Upsilon$.
\end{proof}

\begin{computation}\label{computation-lP=3+III-part2}
In the notation of \xref{notation-blowup}, let 
\begin{eqnarray*}
\alpha&=&y_1^3y_3+y_2^2+y_3^2+\delta y_4^3+c y_1^2y_4^2+\epsilon y_1y_3y_4
+y_2\alpha'+\alpha'',
\\
\beta&= &y_4^2+\nu y_2y_3+\textstyle{\frac 1c} y_1y_3+y_2\beta'+\beta'',
\end{eqnarray*}
where $c$, $\nu$, $\delta$, $\epsilon$ are constants such that
$c\neq 0$ and $\epsilon c\neq \delta$.
We assume that the hypothesis of \xref{computations-notation}
are satisfied.
Then the graph $\Delta(H,C)$ has one of
the following forms:
\begin{equation}\label{graphs-computation-lP=3+III-part2}
\vcenter{
\xy
\xymatrix@R=5pt@C=10pt{
\mathrm{a)}&\overset{C}\bullet\ar@{-}[d]&\ovalh{\phantom{P}}\ar@{-}[d]
\\
\underset{C}\bullet\ar@{-}[r] &\circ\ar@{-}[r] &\underset{3}\circ\ar@{-}[r]&\circ \ar@{-}[r]&\circ
}
\endxy
}
\hspace{30pt}
\vcenter{
\xy
\xymatrix@R=5pt@C=13pt{
\mathrm{b)}&&\ovalh{\phantom{P}}\ar@{-}[d]
\\
\underset{C}{\bullet}\ar@{-}[r] &\circ\ar@{-}[r] &\underset{3}\circ\ar@{-}[r]&\circ \ar@{-}[r]&\circ
}
\endxy
}
\end{equation}
where $\ovalh{\phantom{P}}$ is a non-empty 
connected Du Val subgraph.
In the second case the normalization of $H$ is a bijection.
\end{computation}

\begin{proof}
We use the notation of \xref{notation-blowup}.
By Remark \ref{remark-computation-normality}, \ $y_1^jy_2\notin \beta$ for any $j$.
By \ref{claim-1-construction-blowup} we have
$\Xi=2\Xi_1$, where $\Xi_1:=\{y_2=y_4^2+\frac 1c y_1y_3=0\}$.
The first equation modulo the second one can be rewritten in the form
\begin{equation*}
\alpha=y_2^2+y_3^2+\delta y_4^3+\epsilon y_1y_3y_4
+y_2\alpha'+\alpha''.
\end{equation*}

\begin{sclaim}
The point $O\in \tilde H$ is analytically
isomorphic to a hypersurface singularity of the form
\begin{equation*}
\{y_2^2+y_1y_4^3+ \theta y_1^ry_4^2=0\}\subset \CC^3,
\end{equation*}
where
again $\tilde C$ is the $y_1$-axis, $\theta\in \CC$, and
$r\ge 2$.
\end{sclaim}

\begin{proof}
In the affine chart $U_1$
the equations of $\tilde H$ have the following form
\begin{eqnarray*}
\alpha_{U_1}&=&y_2^2+y_1y_3^2+\delta y_1y_4^3+\epsilon y_1y_3y_4
+y_1y_2\alpha_\bullet+y_1^2\alpha_{\blacktriangle},
\\
\beta_{U_1}&=&y_4^2+\nu y_2y_3+\textstyle{\frac 1c} y_3+y_2\beta_\bullet+y_1\beta_{\blacktriangle},
\end{eqnarray*}
where $\alpha_\bullet\in (y_2,y_3,y_4)$,
$\alpha_{\blacktriangle}$, $\beta_{\blacktriangle}\in (y_2,y_3,y_4)^2$, $\beta_{\bullet}\in (y_2,y_4)$.
{}From the second equation we obtain
\begin{equation*}
y_3= -cu (y_4^2+y_2\beta_\circ+y_1\beta_{\scriptscriptstyle \triangle}),
\end{equation*}
where $\beta_\circ\in (y_2, y_4)$,
$\beta_{\scriptscriptstyle \triangle}\in (y_2, y_4)^2$, and $u$ is a unit such that $u(0)=1$.
Consider the ideal
\begin{equation*}
\mathfrak I:=\left(y_1^2y_4^2,\, y_2^3,\, y_1y_2^2,\,
y_1y_2y_4,\, y_1y_4^4\right).
\end{equation*}
Then we can eliminate $y_3$ in the first equation modulo $\mathfrak I$:
\begin{equation*}
\alpha_{U_1}\equiv y_2^2+(\delta -c\epsilon u) y_1y_4^3 \mod \mathfrak I.
\end{equation*}
Thus, for some $v_i\in \CC\{y_1,y_2,y_4\}$, we can write
\begin{equation*}
\alpha_{U_1}=y_2^2+(\unit) y_1y_4^3+v_1y_1^2y_4^2+v_2y_2^3+v_3y_1y_2^2+v_4y_1y_2y_4+v_5y_1y_4^4.
\end{equation*}
Clearly, the last equation
is analytically equivalent to the desired form.
\end{proof}

\begin{scorollary}\label{scorollary-lP=3+III-sing}
Let $\tilde \psi: \tilde H^{\n}\to \tilde H$ be the blowup of $\tilde C$.
Then $\tilde H^{\n}$ coincides with the normalization and
has exactly one singular point which is of type \type{A_1}.
Moreover,
if $r=2$ and $\theta\neq 0$, then
the preimage $\tilde C^{\n}:=\tilde \psi^{-1}(\tilde C)_{\red}$ has two components
and $\tilde C^{\n}\to \tilde C$ is a double cover.
If $\theta=0$, then $\tilde C^{\n}$ is irreducible and $\tilde C^{\n}\to \tilde C$
is a bijection (near $O$).
If $r>2$ and $\theta\neq 0$, then the total
transform of $\tilde C^{\n}$ on the minimal resolution is not a normal crossing
divisor.
\end{scorollary}

\begin{sclaim}\label{claim-new-11-3}
The intersection $\Xi_1\cap \Sing(\tilde H)$ consists of three points:
$O$, $Q$, and the point $O'\in \Xi_1\cap \Lambda\setminus\{O\}=
\{(0:0: -(\delta -c\epsilon)c : \delta -c\epsilon)\}$.
\end{sclaim}

Now to finish the proof of \ref{computation-lP=3+III-part2}
we notice that by Claim \ref{claim-computation-Xi-n} we have $\hat \Xi_1^2=-3$ \
because $\tau^*\Upsilon$ meets $\hat \Xi_1$
transversely.
This completes the proof of \ref{computation-lP=3+III-part2}.
\end{proof}

\begin{computation}\label{computation-lP=3a-III}
In the notation of \xref{notation-blowup}, let 
\begin{eqnarray*}
\alpha&=&y_1^ly_3+y_2^2+y_3^2+\delta y_4^{2k+1}+c y_1^2y_4^2+\epsilon y_1y_3y_4
+y_2\alpha'+\alpha'',
\\
\beta&=&y_4^2+\nu y_2y_3+\lambda y_1y_3+y_2\beta'+\beta'',
\end{eqnarray*}
where
$l\equiv 3\mod 4$,\
\ $k\ge 1$.
We assume that the hypothesis of \xref{computations-notation}
are satisfied, $\lambda$ is general with respect to $\delta$ and $c$, and
if $l>3$, then $c\neq 0$.
Then the preimage of $C$ on the normalization is
irreducible and the graph $\Delta(H,C)$ has
the following form:
\begin{equation}\label{graph-diagram-non-normal}
\vcenter{\hbox{
\xy
\xymatrix@R=10pt@C=19pt{
&\circ\ar@{-}[d]
\\
\underset C \bullet \ar@{-}[r]
&\underset 3\circ\ar@{-}[r]&\circ\ar@{-}[r]&\circ
\\
&\circ\ar@{-}[u]
}
\endxy
}
}
\end{equation}
\end{computation}

\begin{proof}
We use the notation of \xref{notation-blowup}.
By Remark \ref{remark-computation-normality}, \ $y_1^jy_2\notin \beta$ for any $j$.
We also may assume that $\alpha''$ does not contain any terms of the form $y_4^r$.
By \ref{claim-1-construction-blowup} we have
$\Xi=2\Xi_1$, where $\Xi_1:=\{y_2=y_4^2+\lambda y_1y_3=0\}$.
Since $\lambda\neq 0$, by Claim \ref{claim-3-construction-blowup-a} the set
$\Sing(\tilde H)$ is contained in
$\tilde C\cup \{Q\}\cup \Lambda$.

\begin{sclaim}
The intersection $\tilde H\cap \Lambda$ consists of $O$ and two more
distinct points $P_1$ and $P_2$.
Moreover, $\tilde H$ meets $\Lambda$ transversely at $P_1$ and $P_2$ 
and has singularities of type \type{A_1} at these points.
\end{sclaim}
\begin{proof}
Consider the hypersurface $V\subset W$ defined by $\beta=0$.
Let $\tilde V\subset \tilde W$ be its proper transform.
So, $\tilde H=\tilde X\cap\tilde V$.
We have $(\tilde V|_\Pi\cdot \Lambda)_{\Pi}=4$ and the local intersection number at $O$
equals $2$. Since the base locus of the linear system
on $\Pi$ generated by $\tilde V|_{\Pi}$ meets $\Lambda$ only at $O$, the last assertion follows by
Bertini's theorem and Claim \ref{claim-singularities-Lambda}.
\end{proof}

\begin{sclaim}
$\tilde H\ni O$ is a pinch point.
\end{sclaim}
\begin{proof}
In the affine chart $U_1$
the equations of $\tilde H$ have the form
\begin{equation*}
\begin{aligned}
0&=y_1^{(l+1)/4}y_3+y_2^2+y_1 (y_3^2+\delta y_1^{k-1}y_4^{2k+1}+c y_4^2+\epsilon y_3y_4
+y_2\alpha_{\bullet}+y_1\alpha_{\blacktriangle}),
\\
0&=y_4^2+\nu y_2y_3+\lambda y_3+y_2\beta_\bullet+y_1\beta_\blacktriangle,
\end{aligned}
\end{equation*}
where $\beta_\bullet\in (y_2,\, y_3,\, y_4)$,
$\beta_\blacktriangle\in (y_2,\, y_3,\, y_4)^2$.
{}From the second equation we have
\begin{equation*}
y_3= u(y_4^2+y_2\beta_\circ+y_1\beta_{\scriptscriptstyle \triangle}),
\end{equation*}
where $u$ is a unit such that $u(0)=-1/\lambda$ and $\beta_\circ,\, \beta_{\scriptscriptstyle \triangle}\in (y_2,\, y_4)$.
Eliminating $y_3$ we obtain
\begin{multline*}
uy_1^{(l+1)/4} (y_4^2+y_2\beta_\circ+y_1\beta_{\scriptscriptstyle \triangle})+y_2^2+
u^2y_1(y_4^2+y_2\beta_\circ+y_1\beta_{\scriptscriptstyle \triangle})^2+
\\
\delta y_1^ky_4^{2k+1}+c y_1y_4^2+\epsilon u y_1y_4(y_4^2+y_2\beta_\circ+y_1\beta_{\scriptscriptstyle \triangle})
+y_1y_2\alpha_{\bullet}+y_1^2\alpha_{\blacktriangle}=0,
\end{multline*}
{}From this we see that the equation of $\tilde H$ at $O$ can be written in the form
$y_2^2+y_1y_4^2+\cdots=0$,
i.e. $\tilde H\ni O$ is a pinch point.
\end{proof}
Now to finish the proof of \ref{computation-lP=3a-III}
we  notice that by Claim \ref{claim-computation-Xi-n} we have $\hat \Xi_1^2=-3$ \
because $\tau^*\Upsilon$ is reduced and meets $\hat \Xi_1$
transversely.
\end{proof}

\begin{computation}\label{computation-lP=4+III}
In the notation of \xref{notation-blowup}, let 
\begin{eqnarray*}
\alpha&=&y_1^{4l}y_4+y_2^2+y_3^2+\delta y_4^{2k+1}+c y_1^2y_4^2+\epsilon y_1y_3y_4+y_2\alpha'+\alpha'',
\\
\beta&=&y_4^2+\nu y_2y_3+\eta y_1^2y_4+ y_2\beta'+\beta'',
\end{eqnarray*}
where $l,\, k\ge 1$,\ $c,\, \epsilon \in \CC$,
\ $\delta,\, \eta\in \CC^*$, and $\eta$ is general with respect to $\alpha$.
We assume that the hypothesis of \xref{computations-notation}
are satisfied.
Then the graph $\Delta(H,C)$ has one of the following forms:
\begin{equation}\label{graph-diagram-non-normal-lP=4+III}
\vcenter{
\xy
\xymatrix@R=7pt@C=11pt{
&\circ\ar@{-}[r]&\overset {3}\circ\ar@{-}[d]\ar@{-}[r]&\circ
\\
\bullet \ar@{-}[r] &\underset {}\circ\ar@{-}[r]&\circ\ar@{-}[r]&\circ
}
\endxy}
\end{equation}
\end{computation}

\begin{proof}
We use the notation of \xref{notation-blowup}.
In our case $\Xi=2\Xi_1+2\Xi_2$, where $\Xi_1:=\{y_2=y_4=0\}$, \ 
$\Xi_2:=\{y_2=\eta y_1^2+ y_4=0\}$, and
$\Xi_1\cap \Xi_2=\{Q\}$.

\begin{sclaim}\label{claim-construction-blowup-intersection-Lambda}
\begin{enumerate}[leftmargin=20pt]
\item 
$\Sing(\tilde H)\cap \Xi_1=\{O,\, Q\}$.
\item 
$\Sing(\tilde H)\cap \Xi_2=\Xi_2\cap\Lambda \cup \{Q\}$.
\end{enumerate}
\end{sclaim}

\begin{proof}
By Claim \ref{claim-3-construction-blowup-a} we have 
$\Sing(\tilde H)\subset \Lambda\cup \tilde C\cup \{Q\}$.
On the other hand, $Q\notin \Lambda$ and $\Xi_1\cap \Lambda=\{O\}$.
\end{proof}

\begin{sclaim}
$O\in \tilde H$ is a pinch point.
\end{sclaim}

\begin{proof}
In the affine chart $U_1$
the equations of $\tilde H$ have the form
\begin{eqnarray*}
\alpha_{U_1}&=&y_2^2+y_1(y_1^{l-1}y_4+y_3^2+\delta y_4^{2k+1}+c y_4^2+\epsilon y_3y_4+y_2\alpha_\bullet
+y_1\alpha_\blacktriangle),
\\
\beta_{U_1}&=&y_4^2+\nu y_2y_3+\eta y_4+ y_2\beta_\bullet+y_1\beta_\blacktriangle,
\end{eqnarray*}
where 
$\alpha_\blacktriangle\in (y_2,\, y_3,\, y_4)^2$, 
$\beta_\bullet\in (y_2,\, y_4)$, 
$\alpha_\bullet\in (y_2,\, y_3,\, y_4)$,
and $\beta_\blacktriangle\in (y_2,\, y_3,\, y_4)^2$ by Remark \ref{remark-computation-normality}.
{}From $\beta_{U_1}$ we have
\begin{equation*}
y_4= u y_2y_3+ y_2\beta_1+y_1\beta_2,\quad \beta_1\in (y_2),\ \beta_2\in (y_2,y_3)^2,
\ u=\unit.
\end{equation*}
Then we can eliminate $y_4$ from $\alpha_{U_1}$:
\begin{equation*}
y_2^2+y_1y_3^2+
\gamma_1y_1y_2+\gamma_2y_1 +\gamma_3 y_1^2=0,
\end{equation*}
where $\gamma_1\in (y_2,y_3)$, $\gamma_2\in (y_3)^4$, $\gamma_3\in (y_3)^2$.
By completing the square we can put the equation of $\tilde H$ at $O$ to the following form
\begin{equation*}
y_2^2+(\unit )\cdot y_1y_3^2=0. \qedhere
\end{equation*}
\end{proof}

Recall that by Claim \ref{claim-construction-blowup-Du-Val} the surface
$\tilde H^{\n}$ has only Du Val singularities.
As in \cite[9.1.6]{Mori-Prokhorov-IIA-1}
we see that the pair $(\tilde H, \Xi_1+\Xi_2)$
is not lc at $Q$ and lc outside $Q$ and $\tilde C$.
Thus the dual graph $\Delta(H, C)$ has the form
\begin{equation}\label{graph-diagram-non-normal-lP=4+IIIa}
\vcenter{
\xy
\xymatrix@R=1pt@C=23pt{
&&{\ovalv{\phantom{P}$\scriptstyle P$\phantom{P}}}\ar@{-}[r]&\overset {\Xi_2}\circ \ar@{-}[d]
\\
\underset C \bullet\ar@{-}[rr] &
&\underset {\Xi_1}\circ\ar@{-}[r]&\circ\ar@{-}[r]&\circ
}
\endxy
}
\end{equation}
where \ovalh{$\scriptstyle P$} 
is a Du Val subgraph which is not empty (but possibly disconnected).
By Claim \ref{claim-computation-Xi-n} we have $\hat \Xi_2^2=-3$ and $\hat \Xi_1^2= -2$.
Further, 
\begin{equation*}
\Xi_2\cdot (\Xi_1+\Xi_2)=\textstyle\frac 12 \Xi_2\cdot \Pi= -\frac 23,\quad \Xi_1\cdot \Xi_2=\frac 23,
\quad \Xi_2^2=-\frac 43.
\end{equation*}
Then as in the proof of \cite[Lemma 3.8]{Mori-Prokhorov-IIA-1} we have
$\deg \Diff_{\Xi_2}(0)=5/3$. There are two possibilities:
$\Diff_{\Xi_2}(0)=\frac 23 Q+\frac 12 P_1+ \frac 12 P_2$ and
$\Diff_{\Xi_2}(0)=\frac 23 Q+P_1$. 
Hence the singularities of $\tilde H$ on $\Xi_2\setminus \{Q\}$ are
either two points which are of type \type {A_1} 
or one point which is of type \type{D_n} or \type{A_3}.
The second possibility occurs only for some specific 
choice of $\eta$ (when two intersection points $\Lambda\cap \Xi_2$ coincide). 
We obtain \eqref{graph-diagram-non-normal-lP=4+III}.
\end{proof}

\par\medskip\noindent
{\bf Acknowledgments.}
The paper was written during the second author's visits to RIMS, Kyoto University.
The author is very grateful to the institute for
the invitation, support, and hospitality.

% \bibliography{iia} 
% \bibliography{prokho,my_ref}
% \bibliographystyle{alpha}
% \end{document}

\def\mathbb#1{\mathbf#1} \def\bblapr{April}

\end{document}